\documentclass[12pt]{amsart}
\usepackage{amssymb}
\usepackage{amscd}
\usepackage[mathscr]{eucal}

\newtheorem{thm}{Theorem}[section]
\newtheorem{cor}[thm]{Corollary}
\newtheorem{lem}[thm]{Lemma}
\newtheorem{prop}[thm]{Proposition}

\theoremstyle{definition}
\newtheorem{defn}[thm]{Definition}

\theoremstyle{remark}
\newtheorem{rem}[thm]{Remark}
\newtheorem{ntn}[thm]{Notation}
\newtheorem{hyp}[thm]{Hypothesis}
\newtheorem{num}[thm]{}

\def\vp{{\varphi}}
\def\fc{{\scriptstyle \circ}}

\def\Aut{{\rm Aut}}

\def\Inj{{\rm Inj}}
\def\Syl{{\rm Syl}}
\def\Img{{\rm Im}}
\def\Inn{{\rm Inn}}
\def\Out{{\rm Out}}
\def\Hom{{\rm Hom}}

\def\Rep{{\rm Rep}}

\def\Gs{{\mathscr G}}

\def\Ms{{\mathscr M}}
\def\Ps{{\mathscr P}}

\def\Aa{{\mathcal A}}
\def\Bb{{\mathcal B}}
\def\Cc{{\mathcal C}}
\def\Dd{{\mathcal D}}
\def\Ee{{\mathcal E}}
\def\Ff{{\mathcal F}}
\def\Gg{{\mathcal G}}

\def\Pp{{\mathcal P}}

\def\wBb{\widehat{\mathcal B}}
\def\wCc{\widehat{\mathcal C}}
\def\wFf{\widehat{\mathcal F}}

\def\wG{\widehat{G}}

\def\tG{\widetilde{G}}

\begin{document}

\title[Saturated fusion systems with parabolic families]{Saturated fusion systems with parabolic families}

\author{Silvia Onofrei}
\address{Department of Mathematics, The Ohio State University, 100 Mathematics Tower, 231 West $18^{\rm th}$ Avenue, Columbus, Ohio 43210, USA}
\email{onofrei@math.ohio-state.edu}
\date{Wednesday, 31 August 2011}
\thanks{This project was partially supported by the National Security Agency under grant number H98230-10-1-0174.}

\maketitle

\begin{abstract}
Let $G$ be group; a finite $p$-subgroup $S$ of $G$ is a Sylow $p$-subgroup if every finite $p$-subgroup of $G$ is conjugate to a subgroup of $S$. In this paper, we examine the relations between the fusion system over $S$ which is given by conjugation in $G$ and a certain chamber system $\Cc$, on which $G$ acts chamber transitively with chamber stabilizer $N_G(S)$.

Next, we introduce the notion of a fusion system with a parabolic family and we show that a chamber system can be associated to such a fusion system. We determine some conditions the chamber system has to fulfill in order to assure the saturation of the underlying fusion system. We give an application to fusion systems with parabolic families of classical type.
\end{abstract}

\section{Introduction}

In the present paper we elaborate on the connections between fusion systems, chamber systems and parabolic systems.

A fusion system $\Ff$ is a category whose objects are the subgroups of a finite $p$-group $S$; the morphisms are monomorphisms between these subgroups, such that all monomorphisms induced by conjugation in $S$ are included. The saturated fusion systems satisfy extra conditions which model properties of the fusion in a finite group related to the Sylow $p$-subgroups. The first thorough study of fusion systems and saturated fusion systems is due to Puig; see the more recent \cite{puig06}, for example. Applications to algebraic topology came from the introduction of the notion of $p$-local finite groups, having centric linking systems, by Broto, Levi and Oliver \cite{blo3}. The name saturated fusion system is also due to Broto, Levi and Oliver.

Chamber systems were introduced by Tits \cite{titslocal} in the study of local properties of buildings. All chamber systems used in this paper can be considered as simplicial complexes, with the chambers being simplices of maximal dimension. The codimension one faces are assigned types labeled by elements in an index set $I$, and correspond to the panels of the chamber system. Each chamber has $|I|$ faces; two chambers are $i$-adjacent if they have a face of type $i$ in common.

Let $G$ be a chamber transitive group of automorphisms of a chamber system, and let $B$ denote the stabilizer of a chamber and the $G_i, i \in I$ denote its panel stabilizers. If the chamber system is a building, then $B$ is the Borel subgroup of $G$ and $G_i$ are the minimal parabolic subgroups. The family $(B, G_i \;; i \in I)$ forms a parabolic system of rank $|I|$ if $G = \langle G_i\;; i \in I \rangle$ and no proper subset $\{ G_j\;; j \in J \varsubsetneq I \}$ generates $G$.

We now describe the approach and the results contained in this paper. Let $\Aa$ be a diagram of finite groups $(B, G_i, G_{ij}; i, j \in I)$ together with injective group homomorphisms $B \hookrightarrow G_i$ and $G_i \hookrightarrow G_{ij}$, for all $i, j \in I$, such that all squares commute. Assume that $G$ is a faithful completion of $\Aa$. Our first result Theorem \ref{satfs} asserts that, under certain assumptions, the group $G$ has a finite Sylow $p$-subgroup $S$ and the fusion system $\Ff_S(G)$ associated to $G$ is saturated. This theorem is in some sense a generalization of a result due to Broto, Levi and Oliver \cite[Theorem 4.2]{blo4} regarding the fusion system associated to the colimit of a finite tree of finite groups. Although our proof is written in terms of chamber systems, the line of thought closely follows the one given in \cite{blo4}.

We introduce the notion of a parabolic family for a fusion system $\Ff$ over a finite $p$-group $S$. Roughly speaking it consists of a collection of saturated constrained fusion subsystems $\{ \Ff_i, i \in I\}$ which all contain $\Bb$, the normalizer fusion subsystem $N_{\Ff}(S)$, and with the additional property that each subsystem $\Ff_{ij} = \langle \Ff_i, \Ff_j \rangle$ is also saturated and constrained. We construct a certain group $G$, determined by $p'$-reduced $p$-constrained finite groups that realize the constrained fusion systems $\Bb, \Ff_i$ and $\Ff_{ij}$. A chamber system $\Cc$ over $I$ can be associated to $\Ff$ and $G$. Our second main result is Theorem \ref{mythm} which states that under a couple of assumptions, $\Ff$ is the saturated fusion system of $G$ over $S$. For example, we assume that $\Cc^P$, the fixed point set under the action of a $p$-subgroup $P$ of $G$, is connected.

The third achievement of the paper is Theorem \ref{cchat}. This is a reduction result which says that a fusion system $\Ff$ with a parabolic family contains a certain normal fusion subsystem denoted $\wFf$, which also has a parabolic family. If the hypotheses from Theorem \ref{mythm} hold in $\wFf$, then this fusion system is saturated and realized by a normal subgroup of $G$. Also, there is a $2$-covering $\wCc \rightarrow \Cc$ between the associated chamber systems.

Assume now that the subsystems $\Ff_i$ and $\Ff_{ij}$ are realized by suitably chosen $p'$-reduced $p$-constrained extensions of finite groups of Lie type $G_i$ and $G_{ij}$ in characteristic $p$, of rank one and two respectively. In this case, a type (or diagram) $\Ms$ can be associated to $\Ff$. This is a graph whose vertices are labeled by the elements of $I$, the restriction to the nodes $i$ and $j$ is the Coxeter diagram corresponding to $G_{ij}$. Proposition \ref{mtype} generalizes to fusion systems a well known result of Timmesfeld \cite[3.1]{timm85rev}. Specifically, if the type $\Ms$ is classical, $\Ff$ is the fusion system of a group of Lie type extended by diagram and field automorphisms.

\subsection*{Outline of the paper}
We start with a review of some basic results on fusion systems. In Section 3, we recall the standard terminology on chamber systems and parabolic systems. In Section 4 we prove Theorem \ref{satfs}, while in Section 5 we prove Theorem \ref{mythm}. Section 6 is dedicated to the proof of Theorem \ref{cchat}. We finish our paper with an application to chamber systems of type $\Ms$ and their generalization to fusion systems.

\subsection*{Acknowledgements}
We would like to thank the anonymous referee for a very thorough reading of the paper and for numerous suggestions leading to significant improvements in the paper.

\section{Recollections on fusion systems}

For an introduction to fusion systems see \cite{blo2} and \cite{link07}. We review here the basic definitions and a few results needed in the paper.

\subsection*{Background and terminology}

Given two groups $P$ and $Q$, let $\Hom(P,Q)$ denote the set of group homomorphisms from $P$ to $Q$ and let $\Inj(P,Q)$ denote the subset of monomorphisms.

If $P$ and $Q$ are subgroups of a group $G$, then $c_g:P \rightarrow Q$ denotes the map $x \mapsto gxg^{-1}$ induced by conjugation by $g \in G$. We write $^g P = gPg^{-1}$ and $P^g = g^{-1}Pg$. The transporter set of $P$ into $Q$ is $N_G(P,Q) = \lbrace g \in G |\; ^gP \leq Q \rbrace$. Also let $\Hom_G(P,Q) = N_G(P,Q)/C_G(P)$
denote the set of group homomorphisms from $P$ into $Q$ induced by conjugation in $G$. Set $\Aut_G(P) =\ N_G(P)/C_G(P)$ and $\Out_G(P) = \Aut_G(P)/\Inn(P)$.

\begin{defn}\label{dfs}
A {\it fusion system} $\Ff$ over a finite $p$-group $S$ is a category whose objects are the subgroups of $S$ and whose morphism sets $\Hom_{\Ff}(P,Q)$ satisfy the following two conditions:
\vspace{-.2cm}
\begin{list}{\upshape\bfseries}
{\setlength{\leftmargin}{1cm}
\setlength{\labelwidth}{1cm}
\setlength{\labelsep}{0.2cm}
\setlength{\parsep}{0cm}
\setlength{\itemsep}{0cm}}
\item[${\rm(i).}$] $\Hom_S(P,Q) \subseteq \Hom_{\Ff}(P,Q) \subseteq \Inj(P,Q)$;
\item[${\rm(ii).}$] Every $\Ff$-morphism factors as an $\Ff$-isomorphism followed by an inclusion.
\end{list}
\end{defn}

\begin{defn}\label{fsubgroups}
Let $\Ff$ be a fusion system over a finite $p$-group $S$. A subgroup $P$ of $S$ is
\vspace{-.2cm}
\begin{list}{\upshape\bfseries}
{\setlength{\leftmargin}{1cm}
\setlength{\labelwidth}{1cm}
\setlength{\labelsep}{0.2cm}
\setlength{\parsep}{0cm}
\setlength{\itemsep}{0cm}}
\item[${\rm(i).}$] {\it fully $\Ff$-centralized} if $|C_S( P)|\ge|C_ S({\varphi(P)})|$, for all $\varphi \in \Hom_\Ff(P,S)$;
\item[${\rm(ii).}$] {\it fully $\Ff$-normalized} if $|N_S(P)|\ge|N_ S({\varphi(P)})|$, for all $\varphi \in \Hom_\Ff(P,S)$;
\item[${\rm(iii).}$] {\it $\Ff$-centric} if $C_S(\varphi(P)) = Z(\varphi(P))$ for all $\varphi\in\Hom_\Ff(P,S)$;
\item[${\rm(iv).}$] {\it $\Ff$-radical} if $\Inn(P)=O_p(\Aut_{\Ff}(P))$;
\item[${\rm(v).}$] {\it $\Ff$-essential} if $Q$ is $\Ff$-centric and $\Out_\Ff(P)$
has a strongly $p$-embedded\footnote{A proper subgroup $M$ of $\Out_{\Ff}(P)$ is strongly $p$-embedded if $M$ contains a Sylow $p$-subgroup $Q$ of $\Out_\Ff(P)$ such that
$Q \ne 1$ and $^\varphi Q\cap Q=\{1\}$ for every $\varphi\in\,\Out_\Ff(P)\setminus M$.} subgroup.
\end{list}
\end{defn}

\begin{num}\label{nphi}
For $P \leq S$ and $\varphi \in \Hom_{\Ff}(P, S)$ define: $N_\varphi=\{ x \in N_ S (P) \, | \, \varphi \fc c_x \fc \varphi^{-1}\in \Aut _S(\varphi(P)) \} \,$. It is always the case that $P C_ S(P)\le N_\varphi \leq N_S(P)$.
\end{num}

\begin{defn}\cite{blo4}\label{dsfs}
The fusion system $\Ff$ over a finite $p$-group $S$ is {\it saturated} if the following two conditions hold.
\vspace*{-.2cm}
\begin{list}{\upshape\bfseries}
{\setlength{\leftmargin}{1cm}
\setlength{\labelwidth}{1cm}
\setlength{\labelsep}{0.2cm}
\setlength{\parsep}{0cm}
\setlength{\itemsep}{0cm}}
\item[${\rm(I).}$] For all $P \le S$ which are fully $\Ff$-normalized, $P$ is fully $\Ff$-centralized and $\Aut_S(P)$ is a Sylow $p$-subgroup of $\Aut _{\Ff}(P)$.
\item[${\rm(II).}$] If $P \le S$ and $\varphi \in \Hom_{\Ff}(P,S)$ are such that $\varphi(P)$ is fully $\Ff$-centralized, then there is a morphism $\widehat{\varphi} \in \Hom_{\Ff}(N_{\varphi}, S)$ such that $\widehat{\varphi}|_P=\varphi$.
\end{list}
\end{defn}

\begin{num}
Let $\Ff_i$, $i=1,2$ be fusion systems over subgroups $S_i$ of a $p$-group $S$. We say $\Ff_1$ is a {\it fusion subsystem} of $\Ff_2$ and write $\Ff_1 \subseteq \Ff_2$ if $S_1 \leq S_2$ and $\Hom_{\Ff_1}(P,Q) \subseteq \Hom_{\Ff_2}(P,Q)$ for all subgroups $P$ and $Q$ of $S$.
\end{num}

\begin{num}\label{nmlfs}
Let $\Ee$ and $\Ff$ be fusion systems over $S$ with $\Ee \subseteq \Ff$. We say that {\it $\Ee$ is normal in $\Ff$} and write $\Ee \trianglelefteq \Ff$, if for every isomorphism $\vp: P \rightarrow P'$ in $\Ff$ and subgroups $Q, Q'$ of $P$ we have $\vp_{|Q'} \fc \Hom_{\Ee}(Q, Q') \fc \vp^{-1}_{|\vp(Q)} \subseteq \Hom_{\Ee}(\vp(Q), \vp(Q'))$.
\end{num}

\begin{num}
Assume that $S$ is a finite $p$-group and for each $i = 1, \ldots n$ we are given subgroups $S_i \le S$ and fusion systems $\Ff_i$ over $S_i$. Define $\Ff = \langle \Ff_i\;;  i \in I \rangle$ to be the fusion system generated by $\{\Ff_i\;; i \in I \}$, which is the smallest fusion system over $S$ containing each member of the given collection. The fusion system generated by two saturated fusion systems need not be saturated.
\end{num}

\begin{defn}\label{norcenfs}
Let $\Ff$ be a fusion system over a finite $p$-group $S$ and let $P$ be a subgroup of $S$. The {\it normalizer of} $P$ {\it in} $\Ff$ is the fusion system $N_{\Ff}( P)$ on $N_S(P)$ having as morphisms all group homomorphisms $\varphi: Q \rightarrow R$, for $Q, R$ subgroups of $N_S(P)$, for which there exists a morphism $\widehat{\varphi}: PQ \rightarrow PR$ in $\Ff$ satisfying $\widehat{\varphi}(P)=P$ and $\widehat{\varphi}|_Q=\varphi$. If $\Ff = N_{\Ff}(P)$ then we say that $P$ is normal in $\Ff$ and we write $P \trianglelefteq \Ff$. In a saturated fusion system $\Ff$, if $P$ is
fully $\Ff$-normalized then $N_{\Ff}( P)$ is a saturated fusion system on $N_ S(P)$; for a proof of this statement see \cite[Theorem 3.2]{link07} for example.
\end{defn}

\begin{ntn}\label{op}
Let $O_p(\Ff)$ denote the largest normal $p$-subgroup in $\Ff$. It is a standard result that $O_p(\Ff)$ is contained in every $\Ff$-centric $\Ff$-radical subgroup of $S$.
\end{ntn}

\begin{num}\label{alperin}
Alperin's fusion theorem and its refinement Alperin-Goldschmidt theorem hold for saturated fusion systems. The latter is the statement that every morphism in $\Ff$ is a composite of restrictions of automorphisms of $S$, and of automorphisms of fully $\Ff$-normalized $\Ff$-essential subgroups of $S$; see \cite[Theorem 2.8]{stancu06} for a proof. Note that an $\Ff$-essential subgroup has to be a proper subgroup of $S$ because $\Out _{\Ff} (S) = \Aut_{\Ff}(S)/\Aut_S(S)$ is a $p'$-group.
\end{num}

\subsection*{Fusion systems realized by finite groups}\label{fsfingrp}

It is known, due to the work of Leary and Stancu \cite{ls} and Robinson \cite{rob07} that every fusion system on a finite $p$-group $S$ is equal to the fusion system associated to a group $G$ with Sylow $p$-subgroup $S$. In what follows we discuss the relationship between fusion systems and their associated groups.

\begin{num}
We say that a group {\it $G$ has a (finite) Sylow $p$-subgroup} $S$, if $S$ is a finite $p$-subgroup of $G$ and if every finite $p$-subgroup of $G$ is conjugate to a subgroup of $S$. To a group $G$ with Sylow $p$-subgroup $S$ we can associate a category $\Ff_S(G)$ whose objects are the subgroups of $S$ and whose morphisms are $\Hom_{\Ff}(P, Q) = \Hom_G(P,Q)$ whenever $P$ and $Q$ are subgroups of $S$.
\end{num}

\begin{num}
Notice that in general if $G$ has a Sylow $p$-subgroup, it does not follow that a subgroup $H$ of $G$ also has a Sylow $p$-subgroup, even if $H$ is normal in $G$. However, every subgroup $H$ of $G$ contains a maximal normal $p$-subgroup, which we shall denote $O_p(H)$.
\end{num}

\begin{defn}\label{drealised}
A fusion system $\Ff$ is {\it realized} by a group $G$ (not necessarily finite) if $G$ contains $S$ as a Sylow $p$-subgroup and $\Ff=\Ff_S(G)$.
\end{defn}

\begin{num}\label{fgsatfs}
The classical examples of saturated fusion systems are the ones coming from finite groups. If $G$ is a finite group and $S$ is a Sylow $p$-subgroup of $G$, then $\Ff _S(G)$ is a saturated fusion system. A subgroup $P \leq S$ is fully $\Ff_S(G)$-centralized if and only if $C_S(P) \in \Syl_p(C_G(P))$, while $P$ is fully $\Ff_S(G)$-normalized if and only if $N_S(P) \in \Syl_p(N_G(P))$. For proofs see \cite[Proposition 1.3]{blo2}. Further, a subgroup $P \leq S$ is $\Ff$-centric if and only if $P$ is $p$-centric in $G$, and $P$ is $\Ff$-radical if and only if $O_p(N_G(P)/PC_G(P))=1$.
\end{num}

There are examples of fusion systems that cannot be realized by a finite group, they are called exotic. On the other side, there are examples where one can always construct a finite group with $p$-local structure equivalent to the given fusion system. This is the case for constrained fusion systems. The fusion system $\Ff$ is said to be {\it constrained} if $O_p(\Ff)$ is $\Ff$-centric. We first mention a useful property of the saturated constrained fusion systems.

\begin{prop}\label{constrcen}
Let $\Ff$ be a saturated fusion system over a $p$-group $S$, and let $P$ be a fully $\Ff$-normalized subgroup of $S$. If $N_{ \Ff}( P)$ is constrained then $O_p ( N_{ \Ff}( P))$ is $\Ff$-centric.
\end{prop}

\begin{proof}
Let $P$ be a fully $\Ff$-normalized subgroup of $S$. Set $Q=O_p ( N_{ \Ff}( P))$  and let $\vp \in \Hom_{\Ff}(Q,S)$. As $P$ is fully $\Ff$-normalized, there exists an $\Ff$-morphism $\phi: N_S(\vp(P)) \rightarrow N_S(P)$; see \cite[Lemma 2.6]{link07}. Then $\phi \fc \vp$ maps $P$ to $P$ and maps $Q$ into $N_S(P)$, since $\vp(Q) \le N_S(\vp(P))$. But $Q=O_p(N_{\Ff}(P))$ and $\phi \fc \vp$ fixes $Q$. Since $C_S(\vp(Q)) \le N_S(\vp(P))$, we get $\phi(C_S(\vp(Q))) \le C_S(Q)$. As $N_{\Ff}(P)$ is constrained, $Q$ is centric in $N_{\Ff}(P)$, so $\phi(C_S(\vp(Q))) \le Q$. Then $C_S(\vp(Q)) \le \phi^{-1}(Q) = \vp(Q)$.
\end{proof}

Any constrained fusion system was proven to come from a finite group by Broto, Castellana, Grodal, Levi and Oliver.

\begin{thm}\cite[Theorem 4.3]{bcglo1}\label{thmconstraint}
Let $\Ff$ be a saturated constrained fusion system on a finite $p$-group $S$ and set $Q = O_p(\Ff)$. Then there exists a unique up to isomorphism, finite group $G$ having $S$ as a Sylow $p$-subgroup and such that $Q\trianglelefteq G$, $C_G(Q)=Z(Q)$ and $\Ff= \Ff_S (G)$. Furthermore $G/Z(Q)\simeq\Aut_{\Ff}(Q)$.
\end{thm}

A stronger uniqueness property was shown in \cite[2.5]{afs1}. It asserts that if $G_1$ and $G_2$, two $p'$-reduced $p$-constrained finite groups, realize a saturated constrained fusion system $\Ff$, then there is an isomorphism $\vp: G_1 \rightarrow G_2$ with $\vp_{|S} = {\rm Id}_S$.

A result from elementary group theory which will be useful later.

\begin{lem}\label{pprimep}
Suppose $G$ is $p'$-reduced $p$-constrained finite group and let $S$ be a Sylow $p$-subgroup of $G$. If $H$ is an overgroup of $S$ in $G$ then $H$ is $p'$-reduced $p$-constrained.
\end{lem}

\begin{proof} The Lemma is a direct consequence of the fact that a group $G$ is $p'$-reduced $p$-constrained if and only if $C_G(O_p(G)) \le O_p(G)$.
\end{proof}

In our construction, we will also use the following result due to Aschbacher.

\begin{prop}\cite[1.1]{afs3}\label{subconstr}
Let $\Ff$ be a saturated constrained fusion system on a finite $p$-group $S$ and let $G$ be a finite group that realizes $\Ff$. Assume that $\Ee$ is a saturated constrained subsystem of $\Ff$ on $S$. Then there exists an overgroup $H$ of $S$ in $G$ with $\Ee=\Ff_S(H)$.
\end{prop}

\section{Chamber Systems}

Chamber systems were introduced by Tits \cite{titslocal}; the present treatment follows Scharlau \cite[Sections 1, 6]{scharev} and Ronan \cite[Chapters 1, 4]{ronanbook}. A comprehensive treatment of the theory of coverings of chamber systems can be found in \cite{ronancov}.

\subsection*{Basic notions}

A {\it chamber system} over a set $I$ is a nonempty set $\Cc$ (whose elements are called {\it chambers}) together with a family of equivalence relations $(\sim_i; i\in I)$ on $\Cc$ indexed by $I$. The equivalence classes with respect to $\sim_i$ are called $i$-{\it panels}. Two distinct chambers $c$ and $d$ are called $i$-{\it adjacent} if they are contained in the same $i$-panel; we write $c \sim_i d$. A {\it gallery} of length $n$ connecting two chambers $c_0$ and $c_n$ is a sequence $\gamma: c_0, \ldots, c_n$ of $n+1$ chambers such that $c_{i-1}$ and $c_i$ are $i_j$-adjacent with $i_j \in I$, for all $1 \le j\le n$. The sequence $(i_1, \ldots , i_n)$ is called the {\it type} of the gallery $\gamma$. If each $i_j$ belongs to some subset $J$ of $I$, then $\gamma$ is called a $J$-{\it gallery}.

The chamber system $\Cc$ is {\it connected} (or $J$-{\it connected}) if any two chambers can be joined by a gallery (or $J$-gallery). The $J$-connected components of $\Cc$ are called $J$-{\it residues}. Every $J$-residue is a connected chamber system over the set $J$. The cardinality of the set $I$ is the {\it rank} of the chamber system. The $i$-panels are rank $1$ residues while the chambers are rank $0$ residues.

\begin{num}\label{ccs}
We shall assume from now on that all chamber systems under consideration are connected and of finite rank.
\end{num}

A {\it morphism} $\vp: \Cc \rightarrow \Dd$ between two chamber systems over $I$ is a map defined on the chambers that preserves $i$-adjacency, i.e. if $c, d \in \Cc$ and $c \sim_i d$ then $\vp(c) \sim_i \vp(d)$ in $\Dd$. We denote by $\Aut(\Cc)$ the group of all automorphisms of $\Cc$, where the term automorphism has the obvious meaning. If $G$ is a group of automorphisms of $\Cc$ then $\Cc/G$ inherits a natural structure from $\Cc$, and there is a chamber system $\Cc/G$ over $I$.

\subsection*{Coverings of chamber systems}

A type preserving morphism of chamber systems $\vp: \wCc \rightarrow \Cc$ is called an $m$-{\it covering} if it is surjective and if it maps each rank $m$ residue of $\wCc$ isomorphically onto a rank $m$ residue of $\Cc$. Most of the properties discussed below can be formulated for $m$-coverings; however we shall restrict ourselves to the case when $m=2$, which has greater relevance to the study of buildings as it transpares from \cite{titslocal}.

An {\it elementary $2$-homotopy of galleries} is an alteration from a gallery of the form $\gamma \omega \delta$ to a gallery $\gamma \omega ' \delta$ where $\omega$ and $\omega '$ are galleries (with the same extremities) in a rank $2$ residue. Two galleries are {\it $2$-homotopic} if one can be transformed into the other by a sequence of elementary homotopies. The $2$-homotopy relation is an equivalence relation on the set of galleries.

If $c$ is a chamber in a connected chamber system $\Cc$, a {\it closed gallery based at} $c$ will mean a gallery starting and ending at $c$. The {\it fundamental group} $\pi^2(\Cc, c)$ is the set of $2$-homotopy classes $[\gamma ]$ of closed galleries $\gamma$ based at $c$, together with the binary operation $[ \gamma ] \cdot [ \gamma ' ] = [\gamma \gamma']$ where $\gamma \gamma'$ means $\gamma$ followed by $\gamma'$; using $\gamma^{-1}$ to denote the reversal of $\gamma$, one has $[\gamma]^{-1}=[\gamma^{-1}]$. We call $\Cc$ {\it simply $2$-connected} if it is connected and $\pi^2 (\Cc, c)=1$. If $b$ is another chamber in $\Cc$, $\delta$ is a gallery from $c$ to $b$ and $\gamma$ is a closed gallery based at $c$, the correspondence $[\gamma] \mapsto [\delta ^{-1} \gamma \delta ]$ gives an isomorphism from $\pi^2 (\Cc, c)$ to $\pi^2 (\Cc, b)$. Given a morphism $\vp : (\Cc, c) \rightarrow (\Dd, d)$ with $\vp(c)=d$ there is an induced map $\vp_*: \pi^2(\Cc, c) \rightarrow \pi^2(\Dd, d)$ via $[\gamma] \rightarrow [\vp(\gamma)]$, which is a group homomorphism, and if $\vp$ is a $2$-covering then $\vp_*$ is injective; see \cite[Exercise 1, Chp. 4]{ronanbook}.

To put it briefly, $2$-coverings of chamber systems have similar properties to the topological covers, with the appropriate adjustments on homotopy and simple connectivity; for details see \cite{ronancov} and \cite[4.2]{ronanbook}. In particular we mention a few useful results, which are known from the covering theory of topological spaces.

\begin{lem}\cite[Lemma 4.4]{ronanbook}\label{ulp}
Let $\vp: \wCc \rightarrow \Cc$ be a $2$-covering. Given a gallery $\gamma$ in $\Cc$ starting at some chamber $c$, and given a chamber $\widehat{c} \in \vp^{-1}(c)$, there is a unique gallery $\widehat{\gamma}$ in $\wCc$ starting at $\widehat{c}$ with $\vp(\widehat{\gamma}) = \gamma$.
\end{lem}

\begin{prop}\cite[Theorem 4.8]{ronancov}\label{lifthom}
Let $\vp: (\wCc, \hat{c}) \rightarrow (\Cc, c)$ be a $2$-covering and let $f: (\Dd, d) \rightarrow (\Cc, c)$ be a morphism of chamber systems. The group $f_*(\pi^2(\Dd, d))$ is a subgroup of $\vp_*(\pi^2(\wCc, \hat{c}))$ if and only if there is a unique chamber systems morphism $\widehat{f}: (\Dd, d) \rightarrow (\wCc, \hat{c})$ such that $\vp \fc \widehat{f} = f$.
\end{prop}

Since the results of Proposition \ref{lifthom} will be used several times in the remainder of the section, we provide a proof.

\begin{proof}
If a morphism $\widehat{f}: (\Dd, d) \rightarrow (\wCc, \hat{c})$ exists then for any $[\gamma] \in \pi^2 (\Dd, d)$ we have
$$f_*([\gamma]) = (\vp \fc \widehat{f})_*([\gamma]) = [(\vp \fc \widehat{f}) (\gamma)] = [\vp(\widehat{f}(\gamma))] =\vp_* [\widehat{f}(\gamma)] \in \vp_*(\pi^2(\wCc, \hat{c}))$$
and clearly $f_* (\pi^2(\Dd, d)) \le \vp_*(\pi^2(\wCc, \hat{c}))$.

Conversely, we may assume that $f_*(\pi^2(\Dd, d)) \le \vp_*(\pi^2(\wCc, \hat{c}))$.

{\it Step 1:} we construct a map $\widehat{f}: (\Dd, d) \rightarrow (\wCc, \hat{c})$ with $\vp \fc \widehat{f} = f$.\\
Let $d'$ be a chamber in $\Dd$, distinct from $d$. Let $\gamma_1$ and $\gamma_2$ be two galleries from $d$ to $d'$ in $\Dd$. By Lemma \ref{ulp}, the galleries $f(\gamma_1)$ and $f(\gamma_2)$ from $c = f(d)$ to $f(d')$ have unique lifts to galleries $\widehat{\gamma}_1$ and $\widehat{\gamma}_2$ in $\wCc$ starting at $\hat{c}$, where $\vp(\widehat{\gamma}_i) = f(\gamma_i)$ for $i=1,2$. Next consider the closed gallery $\gamma_1 \gamma^{-1}_2$ based at $d$. Then $f(\gamma_1\gamma^{-1}_2)$ is a closed gallery in $\Cc$ based at $f(d) =c$, whose $2$-homotopy class lies in $\vp_*(\pi^2(\wCc, \hat{c}))$ according to our assumption. Hence $f(\gamma_1 \gamma^{-1}_2)$ lifts uniquely to a closed gallery based at $\hat{c}$, and therefore $\widehat{\gamma}_1$ and $\widehat{\gamma}_2$ have the same endchamber, say $\hat{c}'$. Define $\widehat{f}(d') = \hat{c}'$. The preceding argument shows that $\widehat{f}: \Dd \rightarrow \wCc$ is well-defined.

{\it Step 2:} we show that $\widehat{f}$ is a morphism of chamber systems.\\
Assume that $d$ and $d'$ are $i$-adjacent in $\Dd$. Then $f(d)$ and $f(d')$ are $i$-adjacent and since $\vp$ is isomorphic on rank-$1$ residues it follows that $\vp^{-1}(f(d))=\widehat{f}(d)$ and $\vp^{-1}(f(d')) = \widehat{f}(d')$ are $i$-adjacent, thus $\widehat{f}$ is indeed a chamber systems morphism.

{\it Step 3:} we prove that the morphism $\widehat{f}$ is unique.\\
Assume that $\widehat{f}_1$ and $\widehat{f}_2$ are two chamber systems morphisms, constructed as in {\it Step 1}. Let $d'$ be a chamber in $\Dd$ that is $i$-adjacent to $d$. Then $\widehat{f}_1(d') \sim_i \widehat{f}_1(d) = \widehat{f}_2(d) \sim_i \widehat{f}_2(d')$ in $\wCc$, but $f(d)$ and $f(d')$ are $i$-adjacent in $\Cc$ and since $\vp$ is an isomorphism on rank $1$ residues it follows that $\widehat{f}_1(d')=\widehat{f}_2(d')$. Since $\Dd$ is connected (recall \ref{ccs}) the result follows.
\end{proof}

We record the following direct consequence of Proposition \ref{lifthom}.

\begin{cor}\cite[Proposition 6.1.7]{scharev}\label{liftcrit}
Let $f$ be an automorphism of the chamber system $\Cc$ and let $\vp: (\wCc, \hat{c}) \rightarrow (\Cc, c)$ be a $2$-covering. Given $\hat{c}' \in \vp^{-1}(f(c))$ there exists a unique automorphism $g$ of $\wCc$ satisfying $g(\widehat{c}) = \hat{c}'$ with $\vp \fc g = f \fc \vp$ if and only if $(f \fc \vp)_* (\pi^2(\wCc, \hat{c})) = \vp_* (\pi^2(\wCc, \hat{c}'))$.
\end{cor}

\begin{ntn}\label{deck}
We let $\Aut(\vp)$ denote the group of deck transformation of the $2$-covering $\vp: (\wCc, \hat{c}) \rightarrow (\Cc, c)$; thus the automorphism $g: \wCc \rightarrow \wCc$ is an element of $\Aut(\vp)$ if $\vp \fc g = \vp$.
\end{ntn}

The following result characterizes $2$-coverings of $\Cc$ which correspond to normal subgroups of the fundamental group $\pi^2(\Cc, c)$; such $2$-coverings are called normal in topology.

\begin{prop}\label{norcov}
Let $\vp: (\wCc, \hat{c}) \rightarrow (\Cc, c)$ be a $2$-covering.
\begin{list}{\upshape\bfseries}
{\setlength{\leftmargin}{1cm}
\setlength{\labelwidth}{1cm}
\setlength{\labelsep}{0.2cm}
\setlength{\parsep}{0.5ex plus 0.2ex minus 0.1ex}
\setlength{\itemsep}{0cm}}
\item[${\rm(i).}$] $\vp_*(\pi^2 (\wCc, \hat{c})) \trianglelefteq \pi^2(\Cc, c)$ if and only if for each chamber $\hat{c}' \in \vp^{-1}(c)$ there is a deck transformation $g \in \Aut(\vp)$ with the property $g(\hat{c})=\hat{c}'$.
\item[${\rm(ii).}$] $\Aut(\vp) \simeq N_{\pi^2(\Cc, c)}(\vp_*(\pi^2(\wCc, \hat{c})))/\vp_*(\pi^2(\wCc, \hat{c}))$.
\item[${\rm(iii).}$] $\Cc \simeq  \wCc/\Aut(\vp)$ if and only if $\vp_*(\pi^2(\wCc, \hat{c})) \trianglelefteq \pi^2(\Cc, c)$.
\end{list}
\end{prop}

\begin{proof}
(i). Set $N(H) = N_{\pi^2(\Cc, c)}(\vp_*(\pi^2(\wCc, \hat{c})))$ with $H=\vp_*(\pi^2(\wCc, \hat{c}))$. Observe that changing the base point $\hat{c} \in \vp^{-1}(c)$ to $\hat{c}' \in \vp^{-1}(c)$ corresponds to conjugating $H$ by an element $[\gamma] \in \pi^2(\Cc, c)$, where $\gamma$ is a closed gallery based at $c$ that lifts to a gallery $\widehat{\gamma}$ from $\hat{c}$ to $\hat{c}'$. Thus $[\gamma]$ is in $N(H)$ if and only if $\vp_*(\pi^2(\wCc, \hat{c})) = \vp_*(\pi^2(\wCc, \hat{c}'))$, which by Corollary \ref{liftcrit} is equivalent to the existence of a (unique) deck transformation $g \in \Aut(\vp)$ with $g(\hat{c})=\hat{c}'$.

(ii). Define the morphism $\Phi: N(H) \rightarrow \Aut(\vp)$ given by $\Phi([\gamma]) = g$ where $[\gamma]$ is sent to the deck transformation $g$ which takes $\hat{c}$ to $\hat{c}'$ ($\gamma$ is as in part (i), which means that the unique lift $\widehat{\gamma}$ of $\gamma$ at $\hat{c}$ is a gallery from $\hat{c}$ to $\hat{c}'$). There exists a unique such deck transformation $g$ by Proposition \ref{lifthom} and Corollary \ref{liftcrit}.  Observe that $\Phi$ is a group homomorphism, for if $\gamma'$ is another closed gallery based at $c$, with $\Phi([\gamma'])=g'$ and such that $g'(\hat{c})=\hat{c}''$ then $\gamma  \gamma'$ lifts to $\widehat{\gamma} g(\widehat{\gamma}')$ a gallery from $\hat{c}$ to $g(\hat{c}'')=g(g'(\hat{c}))$, given that $\widehat{\gamma}'$ is the unique lift of $\gamma'$ at $\hat{c}$. Hence $g\fc g'$ is the deck transformation corresponding to $[\gamma] \cdot[\gamma']$. For $g \in \Aut(\vp)$, let $\widehat{\gamma}$ be the gallery from $\hat{c}$ to $\hat{c}' = g(\hat{c})$. (Note here that $\wCc$ is assumed to be connected, see \ref{ccs}). Then for $\gamma = \vp(\widehat{\gamma})$, $[\gamma]$ lies in the preimage of $g$. So $\Phi$ is surjective. The kernel of $\Phi$ consists of those $2$-homotopy classes $[\gamma]$ which lift to closed galleries in $\wCc$, that is the elements of $H$.

(iii). Since $\vp_*(\pi^2(\wCc, \hat{c})) \trianglelefteq \pi^2(\Cc, c)$, it follows by part (i) of this Proposition, that the group of deck transformations $\Aut(\vp)$ acts chamber transitively on $\wCc$.
\end{proof}

\begin{num}
Let $\vp: (\wCc, \hat{c}) \rightarrow (\Cc, c)$ be a $2$-covering of chamber systems. A {\it lift} of $f \in \Aut(\Cc)$ corresponding to $\vp$, is a morphism $\widehat{f} \in \Aut(\wCc)$ with $\vp \fc \widehat{f} = f \fc \vp$.
\end{num}

\begin{defn}\label{univcov}
A $2$-covering $\vp: (\wCc, \hat{c}) \rightarrow (\Cc, c)$ is called {\it universal} if whenever $\psi: (\Cc',c') \rightarrow (\Cc, c)$ is a $2$-covering, there exists some $2$-covering $\alpha: (\wCc, \hat{c}) \rightarrow (\Cc', c')$ such that $\psi \fc \alpha = \vp$. It was shown by Tits \cite[5.1]{titslocal} that universal $2$-coverings of chamber systems always exist; they are unique up to isomorphism. A $2$-covering $\vp: (\wCc, \hat{c}) \rightarrow (\Cc, c)$ is universal if and only if $\wCc$ is simply $2$-connected; see \cite[Proposition 4.6]{ronanbook}.
\end{defn}

In the special case when $\wCc$ is simply $2$-connected, an application of Proposition \ref{norcov} and Corollary \ref{liftcrit} gives the following, cf. \cite[Proposition 6.1.8]{scharev}.

\begin{prop}\label{12.13}
Let $\vp: (\wCc, \hat{c}) \rightarrow (\Cc, c)$ be a universal $2$-covering.
\begin{list}{\upshape\bfseries}
{\setlength{\leftmargin}{1cm}
\setlength{\labelwidth}{1cm}
\setlength{\labelsep}{0.2cm}
\setlength{\parsep}{0.5ex plus 0.2ex minus 0.1ex}
\setlength{\itemsep}{0cm}}
\item[${\rm(i).}$]  $\Aut(\vp) \simeq \pi^2(\Cc, c)$ and $\Cc \simeq \wCc /\Aut(\vp)$.
\item[${\rm(ii).}$] For any subgroup $G \le \Aut(\Cc)$ the set of liftings:
$$\wG: = \lbrace \widehat{g} \in \Aut(\wCc): \vp \fc \widehat{g} = g \fc \vp \; \; {\rm for}\;\;{\rm some}\; \; g \in G \rbrace$$
is a subgroup of $\Aut(\wCc)$, and the map $\wG \rightarrow G,\; \widehat{g} \mapsto g$, is a surjective group homomorphism with kernel $\Aut(\vp)$. For any chamber $\hat{c} \in \wCc$, the stabilizer ${\rm Stab}_{\wG}(\hat{c})$ is mapped isomorphically onto ${\rm Stab}_{G}(\vp(\hat{c}))$.
\end{list}
\end{prop}

\subsection*{Parabolic systems and chamber systems}
Next, we shall discuss the relation between chamber systems and chamber transitive groups of automorphisms.

\begin{num}\label{parsys}
Let $G$ be a group and let $I$ be a finite index set of cardinality $n$. Let $B, G_i, i \in I$ be a collection of subgroups of $G$ satisfying the following properties:
\begin{list}{\upshape\bfseries}
{\setlength{\leftmargin}{2cm}
\setlength{\labelwidth}{1cm}
\setlength{\labelsep}{0.2cm}
\setlength{\parsep}{0cm}
\setlength{\itemsep}{0cm}}
\item[$(P_1)$.] $G  = \langle G_i\;; i \in I \rangle  \ne \langle G_j\;; j \in J \varsubsetneq I \rangle$;
\item[$(P_2)$.] $G_i \cap G_j  = B$, for any distinct $i, j \in I$;
\item[$(P_3)$.] $B \ne G_i$ for all $i \in I$;
\item[$(P_4)$.] $\cap_{g \in G}B^g = 1$.
\end{list}
The family $\Ps (G) =(B, G_i\;; i \in I)$ is called a {\it parabolic system of rank $n$ in $G$}.
\end{num}

\begin{num}\label{associatedchamber}
To $\Ps (G)$ we can associate a chamber system $\Cc(\Ps(G))$ as follows. The left cosets of $B$ in $G$ correspond to the chambers. Two chambers $gB$ and $hB$ are $i$-adjacent if $gG_i = hG_i$ where $g, h \in G$ and $i \in I$, whence $h^{-1}g \in G_i$. Since $G = \langle G_i\;; i \in I \rangle$ the chamber system $\Cc$ is connected. The group $G$ acts on $\Cc$ via left multiplication, an element $x \in G$ takes a chamber $gB$ into a chamber $xgB$; if $gB \sim_i hB$ is an $i$-panel, then $gG_i = hG_i$, which implies $xgG_i=xhG_i$ and therefore $xgB \sim_i xhB$. The action of $G$ on $\Cc$ is chamber transitive; it is also faithful because of ($P_4$). Henceforth the group $G$ will be identified with the automorphisms of $\Cc$ induced by $G$.
\end{num}

\begin{ntn}
Set $G_J = \langle G_j\;; j \in J \rangle$ for all nonempty subsets $J \subseteq I$, with the convention that $G_{\emptyset} := B$. In particular $G_I=G$, $G_i =G_{\{i\}}$ and we write $G_{ij}$ for $G_{\{ i, j \}}$. It is a consequence of $(P_2)$ that $G_J$ is the stabilizer in $G$ of the $J$-residue which contains the chamber $B$.
\end{ntn}

\begin{num}
Given $\Ps(G)=(B, G_i\;; i \in I)$ and $\Ps(G') = (B',G'_i\;; i \in I)$, two parabolic systems of rank $n$ over the same set of types; a group homomorphism $\vp:G \rightarrow G'$ is called a {\it 2-morphism of parabolic systems}, if $\vp(B)=B'$, $\vp(G_i)=G_i'$ and $\vp(G_{ij})=G'_{ij}$ for all $i, j \in I$ with $i \ne j$. In particular $\vp(G)=G'$, thus $\vp$ is surjective.
\end{num}

\begin{lem}\label{morparb}
Let $\vp: \Ps (G) \rightarrow \Ps(G')$ be a 2-morphism of parabolic systems with the property that the restriction of $\vp$ to $G_{ij}$ is a group isomorphism for every distinct $i, j \in I$. Then there is an induced $2$-covering $\vp_*: \Cc(\Ps(G)) \rightarrow \Cc(\Ps(G'))$ between the associated chamber systems. Furthermore $\Aut(\vp_*) \cap G= {\rm Ker}(\vp)$ and $\Cc(\Ps(G')) \simeq \Cc(\Ps(G))/{\rm Ker}(\vp)$.
\end{lem}

\begin{proof}
For $gB$ a chamber in $\Cc(\Ps(G))$, define $\vp_*(gB)=\vp(g)B'$, and for an $i$-panel $gG_i$ set $\vp_*(gG_i)=\vp(g)G_i'$. It is straightforward to check that $\vp_*$ is a morphism of chamber systems. Since $\vp$ is surjective it follows that $\vp_*$ is also surjective, and a $2$-covering of chamber systems. Clearly ${\rm Ker}(\vp) \le \Aut(\vp_*) \cap G$, so it remains to show the opposite inclusion. Let $h \in \Aut(\vp_*) \cap G$. Then, according to \ref{deck}, $\vp(h)\vp(g)B'=\vp(g)B'$, for every $g \in G$. It follows that $\vp(h) \in xB'x^{-1}$ with $x \in G'$. Thus $\vp(h) \in \cap_{x \in G'}x B' x^{-1} =1$ by $(P_4)$, and therefore ${\rm Ker}\vp=\Aut(\vp_*) \cap G$. Observe that, for all $g, h \in G$, $\vp_*(gB) = \vp_*(hB)$ if and only if $g^{-1}h \in \vp^{-1}(B') = {\rm Ker}(\vp) B$. Hence the orbits of ${\rm Ker}(\vp)$ in $\Cc(\Ps (G))$ are just the fibres of $\vp_*$ and $\Cc(\Ps(G')) \simeq \Cc(\Ps(G))/{\rm Ker}(\vp)$.
\end{proof}

\subsection*{Colimits and covers}

\begin{num}\label{amalgam}
Let $I$ be a finite index set. A {\it diagram of (finite) groups} is a pair $\Aa = (\Gs, \Upsilon)$ which consists of a collection of finite groups $\Gs = \{ B, G_i, G_{ij}\;; \; i, j \in I, i\ne j \}$ and a family of injective group homomorphisms $\Upsilon= \{ \tau_B^{G_i}: B \rightarrow G_i, \; \tau_{G_i}^{G_{ij}}: G_i \rightarrow G_{ij}\;; \; i, j \in I\}$ such that $\tau_{G_j}^{G_{ij}} \fc \tau_B^{G_j} = \tau_{G_i}^{G_{ij}} \fc \tau_B^{G_i}$, for all $i, j \in I$. A {\it completion of $\Aa$} is a pair $(\tG, \Phi)$ where $\tG$ is a group endowed with a family $\Phi$ of group homomorphisms $\phi: B \rightarrow \tG$, $\phi_i: G_i \rightarrow \tG$ and $\phi_{ij}:G_{ij} \rightarrow \tG$ which commute with the morphisms in $\Upsilon$. The {\it colimit} $(\tG, \Phi)$ of $\Aa$ (also known as the {\it universal completion} of $\Aa$) has the property that if $(\tG', \Phi')$ is another completion, there is exactly one group homomorphism $f: \tG \rightarrow \tG'$ such that $f \fc \phi = \phi'$, $f \fc \phi_i = \phi_i'$ and $f \fc \phi_{ij} = \phi_{ij}'$, for all $i, j \in I$. The group $\tG$ always exists (although it might be the trivial group) and can be regarded as a group whose set of generating symbols is the disjoint union of elements of $G_i$ and with relations given by the disjoint union of those relations holding in $G_{ij}$ that involve only generators from $G_i$ and $G_j$. Thus $\tG$ is the largest group realizing the diagram of groups, any other completion of $\Aa$ is a quotient of $\tG$. A completion for which the morphisms in $\Phi$ are injective is called {\it faithful}. The pair $\mathcal{A}= (\Gs, \Upsilon)$ admits a faithful completion if and only if its universal completion is faithful.
\end{num}

\begin{num}\label{bcovcs}
If $\Ps (G)= (B, G_i\;; i \in I)$ is a parabolic system in $G$ there is a corresponding diagram of groups $\Ps =  \{ B, G_i, G_{ij}\;; \; i, j \in I, i\ne j \}$ with $G_{ij}:=\langle G_i, G_j \rangle$. Let $\tG$ be the universal completion of $\Ps$. For every $i \in I$, the subgroup $G_i$ of $G$ lifts to a subgroup $\tG_i$ of $\tG$, and similarly $B$ lifts to $\widetilde{B}$, thus $\Ps(\tG) = (\widetilde{B}, \tG_i\;; i \in I)$ is a parabolic system in $\tG$.
\end{num}

\begin{thm}\cite[Proposition 6.5.2]{scharev}\label{covcs}
The chamber system $\Cc(\Ps(\tG))$ is the universal $2$-covering of the chamber system $\Cc(\Ps(G))$. In particular, the natural homomorphism $\tG \rightarrow G$ is an isomorphism if and only if $\Cc(\Ps(G))$ is simply $2$-connected.
\end{thm}

\begin{proof}
Set $\widetilde{\Cc} = \Cc(\Ps(\tG))$ and $\Cc = \Cc(\Ps(G))$. Let $f: \widetilde{\Cc} \rightarrow \Cc$ be the $2$-covering defined as in Lemma \ref{morparb}. Let $\vp : \wCc \rightarrow \Cc$ be the universal $2$-covering of $\Cc$, which exists by \cite[5.1]{titslocal}. According to Proposition \ref{12.13}, the group $G \le \Aut(\Cc)$ lifts through $\vp$ to a subgroup $\wG$ of $\Aut(\wCc)$. The action of $\wG$ on $\wCc$ is chamber transitive; this is a direct consequence of Proposition \ref{12.13}(ii) and of Corollary \ref{liftcrit} (or Proposition \ref{norcov}(i)). Let $\hat{c}$ be a chamber of $\wCc$ with $\vp(\hat{c}) =c$, where $c$ is the chamber corresponding to $B$ in $\Ps(G)$. Let $\Ps(\wG)=(\widehat{B}, \widehat{G}_i\;; i \in I)$ be the parabolic system defined by $\hat{c}$ in $\wG$, where $\widehat{B}={\rm Stab}_{\wG}(\hat{c})$ and $\wG_i$ are the stabilizers in $\wG$ of the $i$-panels containing $\hat{c}$. Another application of Proposition \ref{12.13} gives a chamber system isomorphism $\Cc \simeq \wCc /\Aut(\vp)$ and a group isomorphism $\Aut( \vp) \simeq \pi^2 (\Cc, c)$.

The canonical projection $\psi: \wG \rightarrow G$ induces a $2$-morphism of parabolic systems $\Ps(\wG) \rightarrow \Ps(G)$. Moreover, $\psi$ is an isomorphism when restricted to $\widehat{B}, \wG_i, \wG_{ij}$ with $i, j \in I$; see also \cite[Exercise 8, Chp. 4]{ronanbook}. We can identify the subgroups $\widehat{B}, \wG_i$ and $\wG_{ij}$ of $\Ps(\wG)$ with their counterparts in $\Ps(G)$ and we can regard $\Ps(G)$ as a family of subgroups of $\wG$. As $\tG$ is the universal completion of the family $\Ps$, there is a unique surjective homomorphism $\zeta: \tG \rightarrow \wG$ inducing the identity isomorphism on each of the subgroups $B, G_i, G_{ij}$ with $i, j \in I, i \ne j$. This group homomorphism induces a $2$-covering $\zeta_*: \widetilde{\Cc} \rightarrow \wCc$, and since $\wCc$ is a universal $2$-covering, it follows that $\widetilde{\Cc}$ and $\wCc$ are isomorphic.
\end{proof}

\section{Complexes of groups and saturated fusion systems}

In this section, we consider discrete groups that contain finite Sylow $p$-subgroups. In particular, we investigate the case when $G$ is obtained as a completion of a diagram of finite groups and using the properties of an associated chamber system we identify conditions under which the fusion system $\Ff_S(G)$ is saturated.

\begin{hyp}\label{chamber}
Let $\Aa = (\Gs, \Upsilon)$ be a diagram of groups with $\Gs = \{ B, G_i, G_{ij}\;; \; i, j \in I, i\ne j \}$, a collection of finite groups, $\Upsilon= \{ \tau_B^{G_i}: B \rightarrow G_i, \; \tau_{G_i}^{G_{ij}}: G_i \rightarrow G_{ij}\;; \; i, j \in I\}$, a family of inclusion maps, and $I$ a finite index set. Let $G$ denote a completion of $\Aa$, as defined in \ref{amalgam}. We shall assume that this completion is faithful, so $B, G_i, G_{ij}$, with $i, j \in I$ can be regarded as subgroups of $G$. To $\Gs$ and $G$ we associate a chamber system $\Cc$ in the standard way. The chambers of $\Cc$ are the left cosets of $B$ in $G$, two chambers $gB$ and $hB$ are $i$-adjacent if $gG_i = hG_i$. The group $G$ acts chamber transitively, by left multiplication on $\Cc$. Further, we assume that $\Cc^P$ is connected for every finite $p$-subgroup of $G$. Since the empty set is disconnected, it follows that $\Cc^P \ne \emptyset$.
\end{hyp}

\begin{num}\label{repin}
For any pair of groups $H, G$ let $\Rep(H, G) := \Inn(G) \setminus \Inj(H, G)$ and let $[ \alpha ] \in \Rep(H, G)$ denote the class of $\alpha \in \Inj(H, G)$.
\end{num}

\begin{num}
If $H$ is a finite group, we denote by $\Rep(H, \Cc)$ the chamber system whose chambers are the elements of $\Rep(H,B)$. The $i$-panels are simply represented by the elements of
$$\Rep(H,B,G_i) : = \lbrace [\gamma] \in \Rep(H,G_i): \gamma \in \Inj (H,G_i) \;\; {\rm with} \;\; \gamma(H) \le B \rbrace.$$
Two chambers $[\alpha]$ and $[\beta ]$ are $i$-adjacent if $[\tau_B^{G_i} \fc \alpha ] = [\tau_B^{G_i} \fc \beta ]$ in $\Rep(H, G_i)$ which means that $\tau_B^{G_i} \fc \alpha = c_g \fc \tau_B^{G_i} \fc \beta$ for some element $g \in G_i$. In particular, the $i$-panel $[\gamma ]$ contains the chamber $[\alpha]$ if $[\tau_B^{G_i} \fc \alpha ] = [\gamma ]$ in $\Rep(H, G_i)$. Observe that $\Aut(H)$ acts on $\Rep(H, \Cc)$ via ${\vp}\cdot [\alpha]=[\alpha \fc \vp^{-1}]$, and if $H \le G$, then $N_G(H)$ acts on $\Rep(H, \Cc)$ via $g \cdot [\alpha] = [\alpha \fc c_{g^{-1}}]$.
\end{num}

The next result is Lemma $4.1$ from \cite{blo4} written in terms of chamber systems. We shall use the customary notation $\pi_0(X)$ from topology to denote the set of connected components of a space $X$.

\begin{lem}\label{4.1blo}
Let $G$ and $\Cc$ be as in \ref{chamber} with $P$ a $p$-subgroup of $B$. The following hold.
\begin{list}{\upshape\bfseries}
{\setlength{\leftmargin}{1cm}
\setlength{\labelwidth}{1cm}
\setlength{\labelsep}{0.2cm}
\setlength{\parsep}{0.5ex plus 0.2ex minus 0.1ex}
\setlength{\itemsep}{0cm}}
\item[(a).] There is an $N_G(P)$-equivariant isomorphism of chamber systems between $\Cc^P /C_G(P)$ and $\Rep(P, \Cc)_0$, the connected component of $\Rep(P, \Cc)$ which contains $[\tau_P^B]$, where $\tau_P^B$ denotes the inclusion map of $P$ into $B$.
\item[(b).] The natural map $\Phi_P: \pi_0 (\Rep(P, \Cc)) \rightarrow \Rep(P,B,G)$ is a bijection. In particular, $[\alpha]$ lies in the connected component of $[\tau_P ^B]$ if and only if $\alpha \in \Hom_G(P, B)$.
\end{list}
\end{lem}

\begin{proof}
Note that if $K \le G$ and $gK \in (G/K)^P$ then $P \le gKg^{-1}$. Hence the class $[c_{g^{-1}}]$ in $\Rep(P,K)$ of $c_{g^{-1}}: P \rightarrow K$ given by $x \mapsto g^{-1} x g$, depends on the coset $gK$ only. Define the following map
$$f_P: \Cc^P \rightarrow \Rep (P, \Cc)$$
given by $f_P(gB)=[c_{g^{-1}}]$ where $[c_{g^{-1}}] \in \Rep (P, B)$ and $g \in G$.

{\it Step 1}: $f_P$ is a morphism of chamber systems.\\
We have to show that if $gB$ and $hB$ are $i$-adjacent chambers in $\Cc^P$, then $f_P(gB)=[c_{g^{-1}}]$ and $f_P(hB)=[c_{h^{-1}}]$ are $i$-adjacent in $\Rep(P, \Cc)$. Since $h^{-1}g \in G_i$ it follows that $h^{-1}=pg^{-1}$ for some $p \in G_i$ and that $gG_i g^{-1} = hG_i h^{-1}$. Hence $\tau_B^{G_i} \fc c_{h^{-1}} = \tau_B^{G_i} \fc c_p \fc c_{g^{-1}}$ which says that $[c_{h^{-1}}]=[c_{g^{-1}}]$ in $\Rep(P, B,  G_i)$ and therefore $f_P(gB)$ and $f_P(hB)$ are $i$-adjacent in $\Rep(P, \Cc)$.

{\it Step 2}: $\Img(f_P)$ is $\Rep(P, \Cc)_0$, the connected component of $[\tau _P^B]$ in $\Rep(P, \Cc)$.\\
Let $[\alpha] \in \Rep (P, B)$ be $i$-adjacent to some chamber $[c_{g^{-1}}] \in \Img(f_P)$. Then there exists $h \in G_i$ such that $\alpha = c_{h^{-1}} \fc c_{g^{-1}} = c_{(gh)^{-1}}$ and so $[\alpha] \in \Img(f_P)$. This shows that an $i$-panel lies in $\Img(f_P)$ if one of its chambers lies in $\Img(f_P)$ and therefore $\Img(f_P)$ is a union of connected components of $\Rep(P, \Cc)$. But since $\Cc^P$ is assumed to be connected, we obtain that $\Img(f_P)$ is connected and hence a connected component of $\Rep(P, \Cc)$.

{\it Step 3}: $f_P$ is an $N_G(P)$-equivariant map. \\
The group $N_G(P)$ acts on $\Rep (P, \Cc)$ via $g \cdot [\alpha] = [\alpha \fc c_{g^{-1}}]$ and it is easy to see that this action preserves $i$-adjacency. Since $G$ acts on $\Cc$, $N_G(P)$ acts on the chamber subsystem $\Cc^P$ fixed by the action of $P$. If $n \in N_G(P)$ and $gB$ is a chamber in $\Cc^P$ then
$$f_P(n(gB)) = f_P(ngB) = [c_{(ng)^{-1}}] = [c_{g^{-1}} \fc c_{n^{-1}}]= n \cdot [c_{g^{-1}}] = n \cdot f_P(gB)$$
Hence the chamber systems morphism $f_P$ is $N_G(P)$-equivariant. It also follows that $f_P$ is $C_G(P)$-equivariant and there is an induced chamber systems morphism between the quotients $\Cc^P / C_G(P) \longrightarrow \Rep (P, \Cc)/C_G(P).$ But $C_G(P)$-acts trivially on $\Rep(P, \Cc)$ by definition and therefore there is a morphism $\Cc^P / C_G(P) \longrightarrow \Rep (P, \Cc).$

{\it Step 4}: $f_P$ induces an isomorphism of chamber systems $\Cc^P/C_G(P) \simeq \Img(f_P)$.\\
Two chambers $gB$ and $hB$ of $\Cc^P$ have the same image $f_P(gB)=f_P(hB)$ if and only if $[c_{g^{-1}}] = [c_{h^{-1}}]$ in $\Rep(P, B)$. Then $c_{h^{-1}} = c_{y^{-1}} \fc c_{g^{-1}}$ with $y \in B$ and $h^{-1} = (gy)^{-1}z^{-1}$ where $z \in C_G(P)$. Thus $h=zgy$ which shows that $h \in C_G(P) gB$ and therefore $gB$ and $hB$ are in the same $C_G(P)$-orbit of $\Cc$.

Part $(a)$ of the Lemma follows from {\it Steps 1-4}. It remains to show that the map $\Phi_P$ is bijective. Surjectivity is clear. Let $\alpha, \beta \in \Inj(P,B)$ be such that $\Phi_P([\alpha]) = \Phi_P([\beta])$. Then $\beta \fc \alpha ^{-1} = c_g$ for some $g \in G$. So $[\beta \fc \alpha^{-1}] \in \Img (f_{\alpha (P)})$ and thus by {\it Step 2}, $[\beta \fc \alpha^{-1}]$ lies in the same connected component as $[\tau_{\alpha (P)}^B]$ in $\Rep (\alpha (P), \Cc)$. Therefore $[\alpha]$ and $[\beta]$ are in the same connected component in $\Rep (P, \Cc)$. It follows that $\Phi_P$ is injective.
\end{proof}

\begin{lem}\label{step1}
Let $G$ and $\Cc$ be as in \ref{chamber}. If $S$ is a Sylow $p$-subgroup of $B$ then $S$ is a Sylow $p$-subgroup of $G$.
\end{lem}

\begin{proof}
Let $P$ be a finite $p$-subgroup of $G$. Since $\Cc^P$ is nonempty and $G$ is chamber transitive, there is an element $g \in G$ such that $gB \in \Cc^P$. Hence $P \le gBg^{-1}$ and since $gSg^{-1} \in \Syl_p(gBg^{-1})$ this shows that $P$ is $G$-conjugate to a subgroup of $S$.
\end{proof}

\begin{rem}
Henceforth it makes sense to consider the fusion system $\Ff_S(G)$ over $S$ realized by $G$, as defined in \ref{drealised}.
\end{rem}

\begin{lem}\label{step02}
Let $G$ and $\Cc$ be as in \ref{chamber}. Assume that $S$ is a Sylow $p$-subgroup of $B$ and set $\Ff = \Ff_S(G)$. Then every morphism in $\Ff$ is the composite of morphisms $\vp_1, \ldots, \vp_n$ with $\vp_i \in \Ff_S(G_{j_i})$ for some $j_i \in I$.
\end{lem}

\begin{proof}
Let $\vp = c_g \in \Hom_{\Ff}(P, Q)$ for $P, Q \le S$ and $g \in G$. Then $P,\; ^gP \le S \le B$ and thus $P$ fixes the chambers $B$ and $g^{-1}B$ in $\Cc$. So since $\Cc^P$ is connected, there is a gallery $\gamma$ in $\Cc^P$ from $B$ to $g^{-1}B$. Recall that two chambers $xB$ and $yB$ are $i$-adjacent if and only if $x = yg_i$ for some $g_i \in G_i$. Hence we can write $\gamma = (g_0B, g_0g_1B, \ldots, g_0g_1\cdots g_nB)$ with $g_0=1$, $g_0g_1 \cdots g_n = g^{-1}$ and $g_i \in G_{j_i}$ for some $j_i \in I$. Set $\tilde{g}_i:= g_0\cdots g_i$. As $P$ stabilizes the gallery $\gamma$, $P^{\tilde{g}_i} \le B$. Since $S \in \Syl_p(B)$ we can choose $b_i \in B$ such that $P^{\tilde{g}_i b_i} \le S$ for every $1 \le i \le n$. As $P^{\tilde{g}_0} = P \le S$ and $P^{\tilde{g}_n} = P^{g^{-1}} \le S$ we may take $b_0 = 1 = b_n$. Set $h_i:= b_{i-1}^{-1} g_i b_i \in G_{j_i}$ for $i = 1, \ldots n$. Then:
$$h_1 \cdots h_i = (g_1b_1)\cdot (b_1^{-1}g_2b_2) \cdots (b_{i-1}^{-1} g_i b_i) = g_1 \cdots g_ib_i = \tilde{g}_i b_i$$
Hence $P^{h_1 \cdots h_i} \le S$ for all $i \le n$, and $g^{-1} = \tilde{g}_n = h_1 \ldots h_n$. Thus $c_g$ factors as $c_g = c_{h^{-1}_n} \fc \ldots \fc c_{h^{-1}_1}$ with $c_{h^{-1}_i}: P^{h_1 \cdots h_{i-1}} \rightarrow P^{h_1 \cdots h_i}$ a morphism in $\Ff_S(G_{j_i})$. This completes the proof.
\end{proof}

\begin{lem}\label{step2}
Let $G$ and $\Cc$ be as in \ref{chamber}. Let $S$ be a Sylow $p$-subgroup of $B$. Set $\Ff = \Ff_S(G)$. Assume that every subgroup $P$ of $S$ that is essential in $\Ff_S(G_i)$, for $i \in I$, is $\Ff$-centric. Then every morphism in $\Ff$ is a composite of restrictions of morphisms between $\Ff$-centric subgroups.
\end{lem}

\begin{proof}
The $\Ff$-morphism $\vp = c_g \in \Hom_{\Ff} (P, Q)$ is a composite of morphisms $\vp_1, \ldots, \vp_n$ with $\vp_i \in \Ff_S(G_{j_i})$, for $j_i \in I$; see Lemma \ref{step02}. Each group $G_{j_i}$ is finite, $S$ is a Sylow $p$-subgroup of $G_{j_i}$ and therefore the fusion systems $\Ff_i=\Ff_S(G_{j_i})$ are all saturated. By an application of Alperin-Goldschmidt theorem for fusion systems, see \ref{alperin}, we obtain that each $\vp_i$ can be written as a composite of restrictions of automorphisms of $S$ and of automorphisms of fully $\Ff_S(G_{j_i})$-normalized $\Ff_S(G_{j_i})$-essential subgroups of $S$. But $S$ itself is $\Ff$-centric while $\Ff_S(G_{j_i})$-essential subgroups are $\Ff$-centric for all $j_i \in I$, by assumption. Therefore the Lemma is proved.
\end{proof}

\begin{prop}\label{step3}
Let $G$ and $\Cc$ be as in \ref{chamber}. Assume that $S$ is a Sylow $p$-subgroup of $B$ and set $\Ff = \Ff_S(G)$. Let $P$ be any subgroup of $S$ that is $\Ff$-centric and fully $\Ff$-normalized. Suppose that:
\begin{list}{\upshape\bfseries}
{\setlength{\leftmargin}{1cm}
\setlength{\labelwidth}{1cm}
\setlength{\labelsep}{0.2cm}
\setlength{\parsep}{0.5ex plus 0.2ex minus 0.1ex}
\setlength{\itemsep}{0cm}}
\item[${\rm(i).}$]  the group $\Aut_G(P)$ acts chamber transitively on $\Cc^P/C_G(P)$;
\item[${\rm(ii).}$] given any $p$-subgroup $R$ of $\Aut_G(P)$, the chamber subsystem of $\Cc^P/C_G(P)$ fixed by the action of $R$ is connected.
\end{list}
Then $\Aut_S(P)$ is a Sylow $p$-subgroup of $\Aut_{\Ff}(P)$.
\end{prop}

\begin{proof}
Let $P$ be a $\Ff$-centric fully $\Ff$-normalized subgroup of $S$. The group $\Aut_G(P)$ acts on $\Cc^P/C_G(P)$ and according to {\it Steps 3} and {\it 4} from the proof of Lemma \ref{4.1blo}, it also acts on $\Rep (P, \Cc)_0$. Given $\vp \in \Aut_G(P)$, observe that $\vp \cdot [\tau_P^B ] = [\tau_P^B \fc \vp^{-1}] = [\tau_P^B]$ if and only if $\tau_P^B \fc \vp = c_b \fc \tau_P^B$ for some $b \in B$. Hence $\vp = c_{b\;|P}$ and the stabilizer of the chamber $[\tau_P^B]$ is $\Aut_B(P)$. As $P$ is fully $\Ff$-normalized, it is fully $\Ff_B(S)$-normalized and $\Aut_S(P)$ is a Sylow $p$-subgroup of $\Aut_{\Ff_S(B)}(P) = \Aut_B(P)$. To obtain the conclusion, apply the argument from the proof of Lemma \ref{step1}, with $\Cc^P/C_G(P)$ in place of $\Cc$ and with $\Aut_G(P)$ in place of $G$.
\end{proof}

\begin{prop}\label{step45}
Let $G$ and $\Cc$ be as in \ref{chamber}. Assume that $S$ is a Sylow $p$-subgroup of $B$ and set $\Ff = \Ff_S(G)$. Let $P$ be a $\Ff$-centric subgroup of $S$. Assume that if $R$ is a $p$-subgroup of $\Aut_G(P)$ then the chamber subsystem fixed by the action of $R$ on $\Cc^P/C_G(P)$ is connected. Then for any $\vp \in \Hom_{\Ff}(P, S)$ there is a morphism $\overline{\vp} \in \Hom_{\Ff}(N_{\vp}, S)$ with the property that $\overline{\vp}_{|P} = \vp$.
\end{prop}

\begin{proof}
The proof of the present Proposition will be achieved in two steps. We start with some necessary notation. Let $\vp \in \Hom_{\Ff}(P, S)$ where $P$ is $\Ff$-centric and let $N_{\vp}$ be as in \ref{nphi}. Set $K = \Aut_{N_{\vp}}(P) = N_{\vp}/Z(P)$ and observe that since $P \le N_{\vp}$, the subgroup $N_{\vp}$ is also $\Ff$-centric. Consider the map
$$\Gamma  : \Rep(N_{\vp}, \Cc)_0 \longrightarrow \Rep(P, \Cc)_0$$
induced by the restriction $N_{\vp} \rightarrow P$, between the connected components of the inclusion maps $\tau_{N_{\vp}}^B$ and $\tau_P^B$. The map is well-defined by Lemma \ref{4.1blo}(b). Let $\Rep(P, \Cc)^K_0$ denote the chamber subsystem of $\Rep(P, \Cc)_0$ fixed by $K$.

{\it Step 1}: $\Img(\Gamma) = \Rep(P, \Cc)^K_0$.\\
Let $[\psi] \in \Rep(N_{\vp}, \Cc)_0$ and let $h \in N_{\vp}$ and consider the following commutative diagram
\begin{center}
$\begin{CD}
P  @ >{\tau_P^{N_{\vp}}}>>{N_{\vp}} @> {\psi}>> {B}\\
 @V{c_h}VV      @V{c_h}VV   @V{c_{\psi(h)}}VV    \\
P  @ >{\tau_P^{N_{\vp}}}>> {N_{\vp}} @> {\psi}>> {B}\\
\end{CD}$
\end{center}
from which we see that $c_{\psi(h)} = \psi_{|P} \fc c_h \fc \psi^{-1}_{|P}$. Thus the $\Gamma$-image of $[\psi]$ lies in $\Rep(P, \Cc)^K_0$.

We will show that given any chamber $[\alpha]$ in $\Rep(P, \Cc)_0^K$ and any $i$-panel $[\beta]$ in $\Rep(N_{\vp}, \Cc)_0$ with the property that $[\beta_{|P}]$ is a panel of $[\alpha]$, the chamber $[\alpha]$ lies in $\Img (\Gamma)$. This will finish the proof of this Step since $(\Cc^P/C_G(P))^K$ which is isomorphic as a chamber system to $\Rep(P, \Cc)^K_0$, by Lemma \ref{4.1blo}, is assumed to be connected.

Denote $K' = \alpha K \alpha^{-1} \le \Aut(P')$ with $P'=\alpha(P)$. Set $N' = \{ a \in N_B(P') \; | \; c_a \in K' \}$ and observe that $\Aut_{N'}(P') \le K'$. We shall prove that $\Aut_{N'}(P') = K'$. Since $[\alpha]$ is fixed by $K$, for any $h \in N_{\vp}$, $[\alpha \fc c_h] = [\alpha]$ in $\Rep(P,B)$, thus there is an element $b \in B$ such that $c_b = \alpha \fc c_h \fc \alpha^{-1}:P' \rightarrow P'$, using that $c_h \in K$. Therefore $K' \le \Aut_B(P')$ and in fact $K' \le \Aut_{N_B(P')}(P')$. Let $f \in K'$ so there exists $a \in N_B(P')$ with $f = c_a$. This implies $a \in N'$ and $c_a \in \Aut_{N'}(P')$ and therefore $K' \le \Aut_{N'}(P')$. Next we note that $K' = \Aut_Q(P')$, where $Q \in \Syl_p(N')$ which is true since $K'$ is a $p$-group.

Recall that $\beta : N_{\vp} \rightarrow G_i$, and since $[\alpha]$ lies in the $i$-panel $[\beta_{|P}]$, it follows that $\beta_{|P} = c_g \fc \alpha$, for some $g \in G_i$. Given any $z \in Q$, the morphism $c_g \fc c_z \fc c_{g^{-1}} : \beta (P) \rightarrow \beta(P)$ corresponds to conjugation by an element in $gQg^{-1}$. It follows from the definition of $K'$ that for any $z \in Q$ there is an element $y \in N_{\vp}$ such that $c_z = \alpha \fc c_y \fc \alpha^{-1}$. Therefore
$c_g \fc c_z \fc c_{g^{-1}} = c_g \fc \alpha \fc c_y \fc \alpha^{-1} \fc c_{g^{-1}} = \beta \fc c_y \fc \beta^{-1} = c_{\beta(y)}$ with $\beta (y) \in \beta(N_{\vp})$ and it follows that $g Q g^{-1} \le \beta (N_{\vp}) \cdot C_{G_i}(\beta (P)): = H$.

Observe that if $P$ is $\Ff$-centric then $P$ is $\Ff_S(G_i)$-centric and therefore $P$ is $p$-centric in $G_i$. Therefore $\beta(P)$ is $p$-centric in $G_i$. We claim that both $gQg^{-1}$ and $\beta(N_{\vp})$ are Sylow $p$-subgroups of $H$. Notice that $Z(\beta(P)) = \beta(Z(P)) \in \Syl_p(C_{G_i}(\beta(P)))$ so the unique Sylow $p$-subgroup of $C_{G_i}(\beta(P))$ is already in $\beta(N_{\vp})$. Next note that $K = \Aut_{N_{\vp}}(P) = N_{\vp}/Z(P) \simeq K' = \Aut_Q(P') = Q /Z(P')$ so $|Q| = |N_{\vp}| = |\beta(N_{\vp})|$. The claim is proved; $gQg^{-1}$ and $\beta(N_{\vp})$ are Sylow $p$-subgroups of $H$.

It follows that there is an element $h \in C_{G_i}(\beta(P))$ with the property that $h(gQg^{-1})h^{-1} = \beta(N_{\vp})$ so $Q = c_{(hg)^{-1}}(\beta(N_{\vp}))$. Set $\overline{\alpha} = c_{(hg)^{-1}} \fc \beta: N_{\vp} \rightarrow Q$ with $hg \in G_i$. Finally observe that $\overline{\alpha}_{|P} = c_{(hg)^{-1}} \fc \beta_{|P} = c_{g^{-1}} \fc \beta_{|P} = \alpha$ since $h$ centralizes $\beta(P)$. Thus $[\overline{\alpha}] = [\beta]$ in $\Rep(N_{\vp}, B, G_i)$ showing that $[\alpha]$ is in the image of the restriction map.

{\it Step 2}: Given $\vp \in \Hom_{\Ff}(P, S)$ with $P$ is $\Ff$-centric, then there is a morphism $\overline{\vp} \in \Hom_{\Ff}(N_{\vp}, S)$ with $\overline{\vp}_{|P} = \vp$.\\
Consider the chamber $[\vp]$ which lies in $\Rep(P, \Cc)_0^K$ since the automorphisms in $K$ are induced by elements of $N_{\vp}$. By {\it Step 1} there is a morphism $\psi: N_{\vp} \rightarrow B$ with $[\psi] \in \Rep(N_{\vp}, \Cc)_0$ and $\vp = c_{g^{-1}} \fc \psi_{|P}$ for some $g \in B$. We will use the properties of the saturated fusion system $\Ff_S(B)$ and the fact that $c_{g^{-1}}$ is a morphism in this fusion system. We may assume that $\psi(N_{\vp}) \le S$; if this is not the case choose an element $h \in B$ with $h \psi(N_{\vp})h^{-1} \le S$ and replace $\psi$ by $\psi' = c_h \fc \psi$. This is possible since $[\psi] = [\psi']$ in $\Rep(N_{\vp}, B)$.

Next we prove that $\psi(N_{\vp}) \le N_{\vp \fc \psi^{-1}} = N_{c_{g^{-1}}}$ with $g \in B$. Let $z \in \psi(N_{\vp})$ so $z = \psi (y)$ for $y \in N_{\vp}$ and consider $\vp \fc \psi^{-1} \fc c_z \fc \psi \fc \vp^{-1} = \vp \fc c_y \fc \vp^{-1} \in \Aut_S(\vp(P))$ since $y \in N_{\vp}$. Hence $z \in N_{\vp \fc \psi^{-1}}$.

Thus $\vp \fc \psi^{-1} = c_{g^{-1}}$ extends to a map $\chi: \psi(N_{\vp}) \rightarrow S$ in $\Ff_S(B)$ and therefore $\overline{\vp}:= \chi \fc \psi : N_{\vp} \rightarrow S$ is such that $\chi \fc \psi (y) = \vp (y)$ for all $y \in P$. This ends the proof of the second saturation condition for $\Ff_S(G)$.
\end{proof}

After assembling these results we obtain the following generalization to chamber systems of \cite[Theorem 4.2]{blo4}. Broto, Levi and Oliver considered in their construction the case when $\Cc$ was a finite tree of finite groups.

\begin{thm}\label{satfs}
Let $G$ and $\Cc$ be as in \ref{chamber}. Assume that $S$ is a Sylow $p$-subgroup of $B$ and set $\Ff = \Ff_S(G)$. Suppose the following hold.
\begin{list}{\upshape\bfseries}
{\setlength{\leftmargin}{1cm}
\setlength{\labelwidth}{1cm}
\setlength{\labelsep}{0.2cm}
\setlength{\parsep}{0.5ex plus 0.2ex minus 0.1ex}
\setlength{\itemsep}{0cm}}
\item[(a).] If $P$ is a subgroup of $S$ that is $\Ff$-centric and fully $\Ff$-normalized then $\Aut_G(P)$ acts chamber transitively on $\Cc^P/C_G(P)$.
\item[(b).] If $P$ is a subgroup of $S$ that is $\Ff$-centric and if $R$ is a $p$-subgroup of $\Aut_G(P)$ then $\left ( \Cc^P/C_G(P)\right )^R$ is connected.
\item[(c).] If $P \le S$ is an essential $p$-subgroup of $\Ff_S(G_i)$, for $i \in I$, then $P$ is $\Ff$-centric.
\end{list}
Then $\Ff$ is a saturated fusion system over $S$.
\end{thm}

\begin{proof}
Our proof is a compilation of the last two Propositions and three Lemmas. After we obtain that $S$ is a Sylow $p$-subgroup of $G$, see Lemma \ref{step1}, we show that every morphism in $\Ff$ can be written as a composition of restrictions of morphisms between $\Ff$-centric subgroups, see Lemmas \ref{step02} and \ref{step2}. According to a result of \cite[Theorem 2.2]{bcglo1}, it then suffices to verify the two saturation axioms in \ref{dsfs} for the collection of $\Ff$-centric subgroups only. The first saturation condition is proved in Proposition \ref{step3}, while the second one is proved in Proposition \ref{step45}.
\end{proof}

\section{Parabolic families for fusion systems}

In this section we discuss fusion systems $\Ff$ which contain families of subsystems $\{ \Ff_i\;; i \in I \}$ with certain properties, denoted below (F1 - F4). To such a fusion system we associate a discrete group $G$ and a chamber system $\Cc$, on which the group acts. We give some sufficient conditions $\Cc$ has to fulfill in order to ensure saturation of $\Ff$.

\begin{defn}\label{parffus}
Let $\Ff$ be a fusion system over a finite $p$-group $S$ and set $\Bb = N_{\Ff}(S)$. We say that {\it $\Ff$ has a family of parabolic subsystems} if $\Ff$ contains a collection $\{ \Ff_i; i \in I \}$ of saturated, constrained fusion subsystems, each of essential rank one\footnote{This means that there is one $\Ff_i$-conjugacy class of $\Ff_i$-essential subgroups.} with the following properties:
\vspace*{-.2cm}
\begin{list}{\upshape\bfseries}
{\setlength{\leftmargin}{1.4cm}
\setlength{\labelwidth}{1cm}
\setlength{\labelsep}{0.2cm}
\setlength{\parsep}{0cm}
\setlength{\itemsep}{0cm}}
\item[${\rm(F1).}$] $\Bb$ is a proper subsystem of $\Ff_i$ for all $i \in I$;
\item[${\rm(F2).}$] $\Ff = \langle \Ff_i; i \in I \rangle$ and no proper subset $\{\Ff_j ; j \in J \subset I \}$ generates $\Ff$;
\item[${\rm(F3).}$] $\Ff_i \cap \Ff_j = \Bb$ for any pair of distinct elements $\Ff_i$ and $\Ff_j$;
\item[${\rm(F4).}$] $\Ff_{ij} := \langle \Ff_i, \Ff_j \rangle$ is saturated constrained subsystem of $\Ff$ for all $i, j \in I$.
\end{list}
\end{defn}

\begin{prop}\label{diagfus}
Let $\Ff$ be a fusion system over a finite $p$-group $S$. If $\Ff$ contains a family of parabolic subsystems then there are $p'$-reduced $p$-constrained finite groups $B, G_i, G_{ij}$ that realize $\Bb, \Ff_i, \Ff_{ij}$ respectively (for $i,j \in I$), and injective group homomorphisms $\psi_i:B \rightarrow G_i$, $\psi_{ij}: G_i \rightarrow G_{ij}$ such that $\Aa= \{ (B, G_i, G_{ij}), (\psi_i, \psi_{ij});\; i, j \in I \}$ is a diagram of groups.
\end{prop}

\begin{proof}
First notice that $\Bb$ is a saturated, constrained fusion system. Next, recall that, according to \cite[Theorem 4.3]{bcglo1}, (also see \cite[Theorem I.4.9]{ako}), for every saturated constrained fusion system over a finite $p$-group $S$, there exists a $p'$-reduced $p$-constrained finite group, unique up to isomorphism, which realizes the fusion system. Thus we can find such finite groups $B, G_i, G_{ij}, i,j \in I$ with the property that $\Bb = \Ff_S(B)$, $\Ff_i =\Ff_S(G_i)$ and $\Ff_{ij}=\Ff_S(G_{ij})$. Set $U_i = O_p(\Ff_i)$ and let $U_{ij}=O_p(\Ff_{ij})$ for all $i, j \in I$.

The fusion system $\Bb$ is also realized by $N_{G_i}(S)$ and by $N_{G_{ij}}(S)$, see \cite[Proposition 3.8]{lib1}. Then Lemma \ref{pprimep} gives that the groups $N_{G_i}(S)$ and $N_{G_{ij}}(S)$ are $p'$-reduced $p$-constrained. Thus \cite[Theorem I.4.9(b)]{ako} tells us that there exist isomorphisms $B \simeq N_{G_i}(S) \simeq N_{G_{ij}}(S)$ which are the identity on the Sylow $p$-subgroup $S$. Set $B_i = N_{G_i}(S)$ and $B_{ij} = N_{G_{ij}}(S)$. Denote these isomorphisms by $\alpha_i : B \rightarrow B_i$ and by $\alpha_{ij}: B_i \rightarrow B_{ij}$ with $\alpha_{i|S}={\rm Id}_S = \alpha_{ij|S}$. Let $\tau_i: B_i \rightarrow G_i$ be the inclusion map, and set $\psi_i = \tau_i \fc \alpha_i$.

Let $q_{ij}: N_{G_{ij}}(U_i) \twoheadrightarrow \Aut_{\Ff_{ij}}(U_i)$ denote the canonical quotient map. We use the argument in the proof of \cite[1.1]{afs3} to construct a subgroup of $G_{ij}$ that is isomorphic to $G_i$. Observe that $\Aut_{\Ff_i}(U_i) \subseteq \Aut_{\Ff_{ij}}(U_i)$ and introduce the notation $G_i^{(j)}=q^{-1}_{ij}(\Aut_{\Ff_i}(U_i))$. Since $\Ff_i \subseteq \Ff_{ij}$, the group $G_i^{(j)}$ is an overgroup of $S$ in $G_{ij}$ with the property that $\Ff_i = \Ff_S (G_i^{(j)})$. Let $\tau_i^{(j)}: G_i^{(j)} \rightarrow G_{ij}$ be the inclusion map. The group $G_i^{(j)}$ is $p'$-reduced $p$-constrained as follows from Lemma \ref{pprimep}. Using \cite[21.7]{a3t}, we construct isomorphisms $\vp_i^{(j)}: G_i \rightarrow G_i^{(j)}$.

We need the bottom two rows of the following diagram:
$$\begin{CD}
1 @>>> Z(U_{ij}) @ >>> G_{ij} @ > >> \Aut_{\Ff_{ij}}(U_{ij}) @ >>> 1\\
  @.   @.       @AAA                      @.                @.\\
1 @>>> Z(U_i) @ >>> N_{G_{ij}}(U_i) @ > q_{ij} >> \Aut_{\Ff_{ij}}(U_i) @ >>> 1\\
  @.   @|       @AAA                      @AAA                   @.\\
1 @>>> Z(U_i) @ >>> G_i^{(j)} @ >>> \Aut_{\Ff_i}(U_i)  @ >>> 1\\
  @.   @|       @.                    @|                   @.\\
1 @>>> Z(U_i) @ >>> G_i @ >>> \Aut_{\Ff_i}(U_i)  @ >>> 1
\end{CD}$$
where the vertical arrows correspond to inclusions. We also need the restrictions of these bottom two rows to the subgroups $B_{ij}$ and $B_i$:
$$\begin{CD}
1 @>>> Z(U_i) @ >>> B_{ij} @ >>> {\rm Res}^S_{U_i}( \Aut_{\Ff_i}(S)) @ >>> 1\\
  @.   @|       @AA\alpha_{ij}A             @|                   @.\\
1 @>>> Z(U_i) @ >>> B_i @ >>> {\rm Res}^S_{U_i}(\Aut_{\Ff_i}(S)) @ >>> 1
\end{CD}$$

For every pair $i < j$ in $I$ there exist isomorphisms $\vp_i^{(j)}: G_i \rightarrow G_i^{(j)}$ which extend the isomorphisms $\alpha_{ij}:B_i \rightarrow B_{ij}$.

For pairs $i>j$, let $\vp_i^{(j)}: G_i  \rightarrow G_i^{(j)}$ be isomorphisms extending $\alpha_{ji} \fc \alpha_j \fc \alpha^{-1}_i: B_i \rightarrow B_{ij}$, yielding the commutativity of the following diagram:
$$\begin{CD}
B @>\alpha_j>> B_j @ >\tau_j>> G_j @ >\vp_j^{(i)}>> G_j^{(i)} @ >\tau_j^{(i)}>> G_{ij}\\
@| @. @. @.@|\\
B @>\alpha_i>> B_i @ >\tau_i>> G_i @ >\vp_i^{(j)}>> G_i^{(j)} @ >\tau_i^{(j)}>> G_{ij}
\end{CD}$$

Let $\psi_{ij}: G_i \rightarrow G_{ij}$ be $\psi_{ij} = \tau_i^{(j)} \fc \vp_i^{(j)}$. Consequently, we obtain the following diagram of groups $\Aa= \{ (B, G_i, G_{ij}), (\psi_i, \psi_{ij}); i, j \in I \}$.
\end{proof}

We record the following useful fact for further reference:

\begin{num}\label{gijgenerated}
Let $\Ff$ be a saturated constrained fusion system over a finite $p$-group $S$ with the property that $\Ff = \langle \Ff_1, \Ff_2 \rangle$, where $\Ff_1$ and $\Ff_2$ are saturated constrained subsystems over $S$. Let $G, G_1$ and $G_2$ be $p'$-reduced $p$-constrained finite groups that realize $\Ff, \Ff_1$ and $\Ff_2$ respectively, chosen so that $G_1, G_2 \le G$. We claim that $G = \langle G_1, G_2 \rangle$. Let $H: = \langle G_1, G_2 \rangle$ and observe that $\langle \Ff_1, \Ff_2 \rangle \subseteq \Ff_S(H) \subseteq \Ff$. Hence $\Ff_S(H) = \Ff_S(G)$. But $\Ff_S(H)$ is saturated and constrained, also $H$ is $p'$-reduced $p$-constrained (see Lemma \ref{pprimep}). Finally, combine the fact that $H \le G$ with Theorem \ref{thmconstraint} to conclude that $H  = G$. In particular, we remark that $G_{ij} = \langle G_i, G_j \rangle$.
\end{num}

\begin{lem}\label{fusch}
Let $\Ff$ be a fusion system over a finite $p$-group $S$. Assume $O_p(\Ff) = 1$ and that $\Ff$ contains a family of parabolic subsystems. Let $G$ be a faithful completion of the diagram of groups $\Aa$ from \ref{diagfus}. Then the collection of groups $\{ B, G_i\;; i \in I \}$ is a parabolic system in $G$.
\end{lem}

\begin{proof}
If $G$ is a faithful completion of $\Aa$, then we can identify the groups $B, G_i, i \in I$ with subgroups of $G$. We need to check that conditions (P1)-(P4) from \ref{parsys} are fulfilled. Properties (P1) and (P3) are easy consequences of the properties of $\Ff$ and the way $G$ was constructed.

Next we show that given any distinct $i, j \in I$, $G_i \cap G_j = B$. It is clear that $B \subseteq G_i \cap G_j$. It remains to show the opposite inclusion. Observe that $H:=G_i \cap G_j$ is an overgroup of $S$ in $G_i$ and also in $G_j$; and by \ref{pprimep} the group $H$ is $p'$-reduced $p$-constrained. Thus $\Ff_S(H)$ is a saturated constrained fusion system on $S$ and $\Bb \subseteq \Ff_S(H) \subseteq \Ff_i \cap \Ff_j$. Then, using $(F3)$ we obtain that $\Ff_S(H) = \Bb$. But since $B \le H$ and both $H$ and $B$ are $p'$-reduced $p$-constrained finite groups, realizing the same saturated constrained fusion system it follows that $B \simeq H = G_i \cap G_j$, proving (P2).

Set $B_G = \cap_{g \in G}B^g$ and observe that $O_p(B_G)$ is a Sylow $p$-subgroup of $B_G$. If $O_p(B_G) = 1$ then $B_G$ is a $p'$-group. Since $B_G$ is a subgroup of $B$ that is normal in $G$, it follows that $B_G \trianglelefteq G_i$, for every $i\in I$. But this implies that $O_{p'}(G_i) \ne 1$, a contradiction with the fact that $G_i$ is assumed to be $p'$-reduced. So we must have $O_p(B_G)\ne 1$. In this case $O_p(B_G)$ is a $p$-subgroup of $S$ which is normal in every $G_i, \; i\in I$. It follows that $O_p(B_G) \trianglelefteq \Ff_i$, for all $i \in I$, and since $\Ff = \langle \Ff_i \;; i \in I \rangle$, the $p$-subgroup $O_p(B_G)$ is normal in $\Ff$. Hence $O_p(B_G) \le O_p(\Ff) = 1$ and property (P4) holds.
\end{proof}

\begin{defn}\label{cf}
A {\it fusion-chamber system pair}, denoted by $(\Ff, \Cc)$, consists of:
\vspace*{-.2cm}
\begin{list}{\upshape\bfseries}
{\setlength{\leftmargin}{1cm}
\setlength{\labelwidth}{1cm}
\setlength{\labelsep}{0.2cm}
\setlength{\parsep}{0cm}
\setlength{\itemsep}{0cm}}
\item[${\rm(i).}$] a fusion system $\Ff$, with $O_p(\Ff)=1$, which contains a family of parabolic subsystems (as in \ref{parffus});
\item[${\rm(ii).}$]  a chamber system $\Cc = \Cc(G; B, G_i, i \in I)$, with $G$ a faithful completion of the diagram of groups $\Aa$ from \ref{diagfus}.
\end{list}
\end{defn}

\begin{prop}\label{ffsg}
Let $(\Ff, \Cc)$ be a fusion-chamber system pair, and let $S$ denote a Sylow $p$-subgroup of $B$. Suppose that, for any finite $p$-subgroup $P$ of $G$, $\Cc^P$ is connected. Then:
\begin{list}{\upshape\bfseries}
{\setlength{\leftmargin}{1cm}
\setlength{\labelwidth}{1cm}
\setlength{\labelsep}{0.2cm}
\setlength{\parsep}{0.5ex plus 0.2ex minus 0.1ex}
\setlength{\itemsep}{0cm}}
\item[${\rm(i).}$] $S$ is a Sylow $p$-subgroup of $G$;
\item[${\rm(ii).}$] $\Ff$ is the fusion system of $G$ over $S$, denoted $\Ff_S(G)$;
\item[${\rm(iii).}$] Every morphism in $\Ff$ is a composition of restrictions of morphisms between $\Ff$-centric subgroups.
\end{list}
\end{prop}

\begin{proof}
(i). The first part follows from Lemma \ref{step1}.

(ii). It is clear that $\Ff \subseteq \Ff_S(G)$. The opposite inclusion follows from Lemmas \ref{step02} and \ref{fusch}.

(iii). Recall that $\Ff_i$ has essential rank one and the only $\Ff_i$-essential subgroup $E_i$ of $\Ff_i$ must contain $U_i = O_p(\Ff_i)$. But $\Ff_i$ is saturated and constrained, and according to Proposition \ref{constrcen}, the group $U_i$ is $\Ff$-centric, for all $i \in I$. It then follows that $E_i$ is $\Ff$-centric, and the result is obtained by an application of Lemma \ref{step2}.
\end{proof}

The next Proposition is a slight variation of a technical result due to Linckelmann \cite[Proposition 1.6]{link06} and Stancu \cite[Proposition 4.3]{stancu04}, which applies to any fusion system.

\begin{prop}\label{firstsat}
Let $\Ff$ be a fusion system over a finite $p$-group $S$. Assume that:
\begin{list}{\upshape\bfseries}
{\setlength{\leftmargin}{1cm}
\setlength{\labelwidth}{1cm}
\setlength{\labelsep}{0.2cm}
\setlength{\parsep}{0.5ex plus 0.2ex minus 0.1ex}
\setlength{\itemsep}{0cm}}
\item[${\rm(i).}$] $\Aut_S(S)$ is a Sylow $p$-subgroup of $\Aut_{\Ff}(S)$;
\item[${\rm(ii).}$] given an $\Ff$-centric subgroup $P$ of $S$ and $\vp \in \Hom_{\Ff}(P, S)$, there is a morphism $\overline{\vp} \in \Hom_{\Ff}(N_{\vp}, S)$ with the property that $\overline{\vp}_{|P} = \vp$.
\end{list}
If $P$ is $\Ff$-centric and fully $\Ff$-normalized then $\Aut_S(P)$ is a Sylow $p$-subgroup of $\Aut_{\Ff}(P)$.
\end{prop}

\begin{proof}
Let $Q$ be an $\Ff$-centric fully $\Ff$-normalized subgroup of maximal order such that $\Aut_S(Q)$ is not a Sylow $p$-subgroup of $\Aut_{\Ff}(Q)$. Then $Q$ is a proper subgroup of $S$, as it follows from part (i) of the hypothesis. Choose a $p$-subgroup $R$ of $\Aut_{\Ff}(Q)$ such that $\Aut_S(Q)$ is a proper normal subgroup of $R$. Let $\phi \in R \setminus \Aut_S(Q)$. Since $\phi$ normalizes $\Aut_S(Q)$, for every $y \in N_S(Q)$ there is an element $z \in N_S(Q)$ such that $\phi (y u y^{-1}) = z \phi (u) z^{-1}$, for all $u \in Q$. Thus $N_{\phi} = N_S(Q)$. Since $Q$ is $\Ff$-centric, it follows from part (ii) of the hypothesis that $\phi$ extends to $\overline{\phi}: N_{\phi} \rightarrow N_S(Q)$, so that $\overline{\phi} \in \Aut_{\Ff}(N_S(Q))$. Since $\phi$ has $p$-power order, by decomposing $\overline{\phi}$ into its $p$-part and its $p'$-part we may assume that $\overline{\phi}$ has $p$-power order.

Let $\psi: N_S(Q) \rightarrow S$ be a morphism in $\Ff$ such that $\psi (N_S(Q)) = N'$ is fully $\Ff$-normalized. As the order of $N'$ is greater that the order of $Q$ (also observe that $N'$ is $\Ff$-centric), we have that $\Aut_S(N') \in \Syl_p(\Aut_{\Ff}(N'))$. Now $\psi \fc \overline{\phi} \fc \psi^{-1}$ is a $p$-element of $\Aut_{\Ff}(N')$, thus conjugated to an element in $\Aut_S(N')$. Therefore we may choose $\psi$ in such a way that there is $y \in N_S(N')$ satisfying $\psi \fc \overline{\phi} \fc \psi^{-1} (v) = c_y(v)$ for all $v \in N'$. Since $\overline{\phi}_{|Q}=\phi$, it follows that $\psi \fc \overline{\phi} \fc \psi^{-1} (\psi (Q)) = \psi (Q)$ and $y \in N_S(\psi(Q))$. But $Q$ is fully $\Ff$-normalized and since $\psi (N_S(Q)) \subseteq N_S(\psi (Q))$ we have that $\psi(N_S(Q)) = N_S(\psi(Q))$. Hence $\overline{\phi}(u) = \psi^{-1} \fc c_y \fc \psi(u)$ for all $u \in N_S(Q)$. In particular $\phi \in \Aut_S(Q)$ contradicting our choice of $\phi$.
\end{proof}

To this end we can combine the results of this section in the following:
\begin{thm}\label{mythm}
Let $(\Ff, \Cc)$ be a fusion-chamber system pair. Assume the following hold.
\begin{list}{\upshape\bfseries}
{\setlength{\leftmargin}{1cm}
\setlength{\labelwidth}{1cm}
\setlength{\labelsep}{0.2cm}
\setlength{\parsep}{0.5ex plus 0.2ex minus 0.1ex}
\setlength{\itemsep}{0cm}}
\item[${\rm(a).}$] $\Cc^P$ is connected for all $p$-subgroups $P$ of $G$.
\item[${\rm(b).}$] If $P$ is an $\Ff$-centric subgroup of $S$ and if $R$ is a $p$-subgroup of $\Aut_G(P)$, then $(\Cc^P/C_G(P))^R$ is connected.
\end{list}
Then $\Ff = \Ff_S(G)$ is a saturated fusion system over $S$.
\end{thm}

\begin{proof}
Assume that $\Ff$ is a fusion system over a finite $p$-group $S$, which contains a family of parabolic subsystems which fulfill the properties (F1)-(F4) from \ref{parffus}. Let $B, G_i, G_{ij}$, with $i, j \in I$, be $p'$-reduced $p$-constrained finite groups that realize $\Bb, \Ff_i, \Ff_{ij}, i, j \in I$. Under our assumption that $G$ is a faithful completion, the groups $B, G_i, G_{ij}, i, j \in I$ can be regarded as subgroups of $G$, and according to Lemma \ref{fusch}, they form a parabolic system in $G$, in the sense of \ref{parsys}. Let $\Cc = \Cc(G; B, G_i, i \in I)$ be the associated chamber system described in \ref{associatedchamber}. Next, assuming that $\Cc^P$ is connected, for every $P \le G$ and using Proposition \ref{ffsg}, it is obtained that $\Ff$ is realized by $G$, in other words $\Ff = \Ff_S(G)$.

It remains to show that the fusion system $\Ff$ is saturated. First, it is shown that every morphism in $\Ff$ can be written as a composition of restrictions of morphisms between $\Ff$-centric subgroups; this is the result of Proposition \ref{ffsg}(iii). It follows from \cite[Theorem 2.2]{bcglo1}, that it suffices to verify the two saturation axioms in \ref{dsfs} for the collection of $\Ff$-centric subgroups only. The second saturation condition (II) is obtained from Proposition \ref{step45}. Further $\Aut_{\Ff}(S) = \Aut_{\Bb}(S)$ and $S \in \Syl_p(B)$, hence $S/Z(S) = \Aut_S(S) \in \Syl_p(\Aut_{\Bb}(S))$. Therefore the conditions from the hypothesis of Proposition \ref{firstsat} are in place and the first saturation axiom (I) follows. This concludes the proof of the fact that $\Ff$ is a saturated fusion system over $S$.
\end{proof}

\section{A subsystem with a parabolic family}

We show that a fusion system $\Ff$ over $S$ with a family of parabolic subsystems contains a certain saturated subsystem $\wFf$ over $S$ which also has an associated parabolic family. Before proceeding with our construction we shall review some standard facts and properties of certain normal subsystems.
\subsection*{Subsystems of index prime to p}

The material included in this overview appeared elsewhere in the literature, we refer the reader to \cite[Chapter 12]{puigbook} and to \cite[Section 5]{bcglo2} for earlier sources. We will follow the more recent approach from \cite[Section I.7]{ako}.

\begin{prop}\label{puigprop}
Let $\Ff$ and $\Gg$ be fusion systems over $S$ with $\Gg \subseteq \Ff$ and $\Ff$ saturated. Assume that $O^{p'}(\Aut_{\Ff}(Q)) \le \Aut_{\Gg}(Q)$ for every subgroup $Q$ of $S$. Then:
\begin{list}{\upshape\bfseries}
{\setlength{\leftmargin}{1cm}
\setlength{\labelwidth}{1cm}
\setlength{\labelsep}{0.2cm}
\setlength{\parsep}{0.5ex plus 0.2ex minus 0.1ex}
\setlength{\itemsep}{0cm}}
\item[${\rm(i).}$] $\Hom_{\Ff}(Q, S) = \Aut_{\Ff}(S) \fc \Hom_{\Gg}(Q, S)$;
\item[$(ii).$] $Q$ is fully $\Ff$-normalized ($\Ff$-centralized) iff it is fully $\Gg$-normalized ($\Gg$-centralized);
\item[$(iii).$] $Q$ is $\Ff$-centric if and only if it is $\Gg$-centric;
\item[$(iv).$] $Q$ is $\Ff$-essential if and only if it is $\Gg$-essential.
\end{list}
\end{prop}

\begin{proof}
$(i)$. Let $Q$ be a subgroup of $S$. Clearly we have $\Aut_{\Ff}(S) \fc \Hom_{\Gg}(Q, S) \subseteq \Hom_{\Ff}(Q,S)$. It remains to prove that for any morphism $\vp \in \Hom_{\Ff}(Q, S)$ there exist $\alpha \in \Aut_{\Ff}(S)$ and $\zeta \in \Hom_{\Gg}(Q, S)$ with $\vp = \alpha \fc \zeta$. We argue by induction on the index $|S:Q|$. If $S=Q$ then $\vp \in \Aut_{\Ff}(S)$ and we are done. So we may suppose that $|S :Q| >1$, which means that $Q < S$. By a standard argument (see the proof of Theorem A.10 in \cite{blo2}) we can show that it suffices to find the sought decomposition for an automorphism $\vp \in \Aut_{\Ff}(Q)$ of a fully $\Ff$-normalized subgroup $Q$ of $S$. By a general Frattini argument, using the fact that $\Ff$ is saturated and hence $\Aut_S(Q) \in \Syl_p(\Aut_{\Ff}(Q))$ we obtain:
$$\Aut_{\Ff}(Q) = N_{\Aut_{\Ff}(Q)} (\Aut_S(Q)) \cdot O^{p'}(\Aut_{\Ff}(Q)).$$
Thus $\vp = \phi \fc \eta$ with $\phi \in N_{\Aut_{\Ff}(Q)}(\Aut_S(Q))$ and $\eta \in O^{p'}(\Aut_{\Ff}(Q))$. Next observe that
$$N_{\Aut_{\Ff}(Q)} (\Aut_S(Q)) = \{ \rho \in \Aut_{\Ff}(Q)\; |\; N_{\rho}=N_S(Q) \}.$$

Since $Q$ is fully $\Ff$-normalized, it is fully $\Ff$-centralized, and the saturation axiom (II) implies that $\phi$ extends to a map $\widehat{\phi}: N_S(Q) \rightarrow S$ which has the property that $\widehat{\phi}(Q)=\phi(Q)=Q$. Since $Q < N_S(Q) \le S$, the induction hypothesis gives that $\widehat{\phi} = \alpha_{|\widehat{\chi}(N_S(Q))} \fc \widehat{\chi}$ where $\alpha \in \Aut_{\Ff}(S)$ and $\widehat{\chi} \in \Hom_{\Gg}(N_S(Q), S)$. Consequently $\phi = \alpha_{|\widehat{\chi}(Q)} \fc \widehat{\chi}_{|Q}$ has the desired form. Therefore $\vp = \phi \fc \eta = \alpha_{|\widehat{\chi}(Q)} \fc \widehat{\chi}_{|Q} \fc \eta$ with $\widehat{\chi}_{|Q} \fc \eta \in \Hom_{\Gg}(Q,S)$ and $\vp$ has the required form also.

$(ii)$. If $Q$ is fully $\Ff$-normalized then $Q$ is also fully $\Gg$-normalized, given $\Gg \subseteq \Ff$. Conversely, if $Q$ is fully $\Gg$-normalized and $\psi: N_S(Q) \rightarrow S$ is an $\Ff$-morphism, we obtain that $\psi = \alpha \fc \zeta$ with $\alpha \in \Aut_{\Ff}(S)$ and $\zeta \in \Hom_{\Gg}(N_S(Q), S)$. Also $|N_S(\zeta(Q))| \le |N_S(Q)|$. On the other side, we always have $\zeta(N_S(Q)) \le N_S(\zeta(Q))$ and since $\zeta$ is injective, $N_S(\zeta(Q)) = \zeta(N_S(Q))$. Therefore
$$\psi(N_S(Q)) = (\alpha \fc \zeta)(N_S(Q)) = \alpha(N_S(\zeta(Q))) = N_S(\alpha \fc \zeta (Q)) = N_S(\psi(Q))$$
which shows that $Q$ is fully $\Ff$-normalized. The other statement from $(ii)$ can be proved in a similar way.

$(iii)$. Clearly, each $\Ff$-centric subgroup $P$ is also $\Gg$-centric. Conversely, if $P$ is $\Gg$-centric then, by part $(i)$, each $\Ff$-conjugate of $P$ has the form $\alpha (Q)$ with $\alpha \in \Aut_{\Ff}(S)$ and $Q$ a $\Gg$-conjugate of $P$. Then $C_S(\alpha (Q)) \le \alpha (Q)$ since $C_S(Q) \le Q$.

$(iv)$. Observe $O^{p'}(\Aut_{\Ff}(Q)) = O^{p'} ( \Aut_{\Gg}(Q))$ and thus $O^{p'} (\Out_{\Ff}(Q)) = O^{p'}(\Out_{\Gg}(Q))$. By the Frattini argument, a finite group $G$ has a strongly $p$-embedded subgroup if and only if $O^{p'}(G)$ has one. Hence $\Out_{\Ff}(Q)$ has a strongly $p$-embedded subgroup if and only if $\Out_{\Gg}(Q)$ has one. Now the assertion follows from part $(iii)$.
\end{proof}

We use the result of \cite[Theorem I.7.7]{ako} to formulate the following:

\begin{defn}\label{oprime}
Given a saturated fusion system $\Ff$ on a finite $p$-group $S$, we let $O^{p'}(\Ff)$ denote the smallest saturated fusion subsystem of $\Ff$ which has the property that $\Aut_{O^{p'}(\Ff)}(P) \ge O^{p'}(\Aut_{\Ff}(P))$ for every subgroup $P$ of $S$.
\end{defn}

It follows from Proposition \ref{puigprop} that the Frattini argument holds for saturated fusion systems.

\begin{cor}\label{frattini}
If $\Ff$ is a saturated fusion system over a finite $p$-group $S$ then there is a decomposition $\Ff = \langle O^{p'}(\Ff), N_{\Ff}(S) \rangle$.
\end{cor}

If $\Ff=\Ff_S(G)$ then $O^{p'}(\Ff)$ does not necessarily correspond to $O^{p'}(G)$. However, there are particular cases in which the correspondence is attained.

\begin{prop}\label{oprimeconstrained}
Let $\Ff = \Ff_S(G)$ be a saturated constrained fusion system over a finite $p$-group $S$ with $G$ a $p'$-reduced $p$-constrained finite group. Then $O^{p'}(\Ff)=\Ff_S(O^{p'}(G))$.
\end{prop}

\begin{proof}
Let $\Ff$ and $G$ be as in the hypothesis. Since $O^{p'}(G)$ is a normal subgroup of $G$, it follows that $\Ff_S(O^{p'}(G))$ is a normal saturated fusion subsystem of $\Ff$. On the other side, $\Ff$ is constrained and Aschbacher's Theorem \cite[Theorem 1]{afs1} asserts the existence of a unique normal subgroup $O_G$ of $G$ with $O^{p'}(\Ff)=\Ff_S(O_G)$.

Set $\Ff'=\Ff_S(O^{p'}(G))$ and $\wFf=O^{p'}(\Ff)$. The inclusion $\Ff' \subseteq \wFf$ follows from $O^{p'}(G) \le O_G$. To prove that $\wFf \subseteq \Ff'$, let $\vp \in \Aut_{\Ff}(Q)$, $Q \le S$ be a $p$-element. Then $\vp = c_g$ for some $g \in N_G(Q)$. Then we may choose $g$ to be a $p$-element as well and so $g \in O^{p'}(G)$. Hence $\vp = c_g \in \Aut_{\Ff'}(Q)$. This proves $O^{p'}(\Aut_{\Ff}(Q)) \le \Aut_{\Ff'}(Q)$ and thus $\wFf \subseteq \Ff'$.
\end{proof}

\subsection*{A reduction result}

For the rest of this section we shall assume that $(\Ff, \Cc)$ is a fusion-chamber system pair as defined in \ref{cf}. According to Lemma \ref{fusch}, the family $(B, G_i\;; i \in I)$ is a parabolic system in $G$, a faithful completion of the diagram of groups $\Aa$ from Proposition \ref{diagfus}. Thus $B \le G_i$ for all $i \in I$, and by \ref{gijgenerated} we also have $G_{ij} = \langle G_i, G_j \rangle$.

\begin{ntn}\label{fhat}
The fusion systems $\Ff_i = \Ff_S(G_i)$ and $\Ff_{ij} = \Ff_S(G_{ij})$, $i,j \in I$, are saturated and constrained with $G_i$, $G_{ij}$ respectively, $p'$-reduced $p$-constrained finite groups. Denote $\wG_i = O^{p'}(G_i)$ and set $\wFf_i = O^{p'}(\Ff_i)$, for $i \in I$. According to Proposition \ref{oprimeconstrained}, $\wFf_i = \Ff_S(\wG_i)$ and the fusion system $\wFf_i$ is saturated and constrained. Define $\wFf = \langle \wFf_i; i \in I \rangle$ and let $\wBb = \langle N_{\wFf_i}(S); i \in I \rangle$.
\end{ntn}

\begin{lem}\label{sathatb}
The fusion subsystem $\wBb$ is saturated and constrained, and $\wBb = N_{\wFf}(S)$. Further $\wBb = \Ff_S(\widehat{B})$ with $\widehat{B} = \langle N_{\wG_i}(S); i \in I \rangle$.
\end{lem}

\begin{proof}
Because $N_{\wFf_i}(S) \subseteq N_{\Ff}(S)$ it follows that $\wBb \subseteq \Bb$. Thus every morphism in $\wBb$ extends to an $\Ff$-automorphism of $S$, and this implies that $\wBb$ is a saturated fusion system. It is obviously constrained since $S$ is a normal $\wBb$-centric subgroup. Next observe that since $N_{\wFf_i}(S) \subseteq N_{\wFf}(S)$ for all $i \in I$, it follows that $\wBb \subseteq N_{\wFf}(S)$. Let now $\vp \in \Hom_{N_{\wFf}(S)}(P,Q)$ for $P, Q \le S$. There exists a morphism $\widehat{\vp} \in \Aut_{\wFf}(S)$ which extends $\vp$. Since $\wFf=\langle \wFf_i; i \in I \rangle$, there are morphisms $\widehat{\psi}_j \in \Aut_{\wFf_i}(S)$, $i \in I$ and $j=1, \ldots n$, such that $\widehat{\vp}=\widehat{\psi}_n \fc \ldots \fc \widehat{\psi}_1$. It follows now that there are subgroups $P=P_0, P_1, \ldots P_n=Q$ of $S$ and morphisms $\psi_j:P_{j-1} \rightarrow P_j$ with $\psi_j={\widehat{\psi}_j}|_{P_{j-1}}$ and $\vp = \psi_n \fc \ldots \fc \psi_1$. Because $\psi_j$ is the restriction of an $\wFf_i$-automorphism of $S$, it follows that $\psi_j$ is a morphism in $N_{\wFf_i}(S)$; this proves that $\vp \in \Hom_{\wBb}(P,Q)$.

Since $\wBb \subseteq \Bb$, according to Proposition \ref{subconstr} and Lemma \ref{pprimep}, there exists a $p'$-reduced $p$-constrained subgroup $\widehat{B}$ of $B$ which realizes $\wBb$. Using \cite[Proposition 3.8]{lib1}, we can identify $B$ with $N_{G_i}(S)$, for all $i \in I$. Set $T: = \langle N_{\wG_i}(S); i \in I \rangle \le B$ and notice that $\Ff_S(T)$ is a saturated constrained fusion system. We will show that $\widehat{B} = T$. Since each $\wFf_i = \Ff_S(\wG_i)$ is saturated, another application of \cite[Proposition 3.8]{lib1} gives that $N_{\wFf_i}(S) = \Ff_S(N_{\wG_i}(S))$. Thus $\wBb \subseteq \Ff_S(T)$ and, using Proposition \ref{subconstr}, we can choose $\widehat{B} \le T$. Conversely, let $g \in T$ and $P \le S$ be such that $^gP \le S$. But $g = g_m \cdots g_1$ for $g_i \in N_{\wG_{j_i}}(S)$ and $j_i \in I$. Thus $c_g : P \rightarrow \;^g P$ can be decomposed as $P \rightarrow \; ^{g_1}P \rightarrow \ldots \rightarrow \; ^{g_m \cdots g_1}P= \;^g P$ with $^{g_i \cdots g_1} P \le S$ since each $g_i$ normalizes $S$. Therefore $c_{g_i} : \;^{g_{i-1} \cdots g_1} P \rightarrow \; ^{g_i \cdots g_1} P$ is in fact a morphism in $N_{\wFf_{j_i}}(S) \subseteq \wBb$. Thus $\wBb = \Ff_S(T)$ and Theorem \ref{thmconstraint} together with the fact that $\widehat{B} \le T$ give that $\widehat{B}=T$.
\end{proof}

\begin{lem}\label{sathatg}
The fusion systems $\Gg_i := \langle \wFf_i, \wBb \rangle$ and $\Gg_{ij}: = \langle \Gg_i, \Gg_j \rangle$ are saturated and constrained for all $i, j \in I$.
\end{lem}

\begin{proof}
Since $\Gg_i \subseteq \Ff_i$, it follows that $O_p(\Ff_i)$ is normal in $\Gg_i$, it is $\Gg_i$-centric because it is $\Ff_i$-centric. Thus $\Gg_i$ is a constrained fusion system on $S$; similarly $\Gg_{ij}$ is also constrained.

Let $\wBb \subseteq \Bb \subseteq \Ff_i$ and the corresponding $p'$-reduced $p$-constrained finite groups $\widehat{B} \le B \le G_i$, where $\widehat{B}$ is as in Lemma \ref{sathatb}. We will show that $\Gg_i = \Ff_S(H_i)$ with $H_i : = \wG_i \widehat{B}$, a $p'$-reduced $p$-constrained subgroup of $G_i$. This will suffice to prove the saturation of $\Gg_i$ (according to \ref{fgsatfs}). Clearly $\wFf_i, \wBb \subseteq \Ff_S(H_i)$. Conversely, let $h \in H_i$ and $P \le S$ be such that $^h P \le S$ so $c_h \in \Hom_{\Ff_S(H_i)} (P, S)$. But $h = gb \in \wG_i$ with $g \in \widehat{G}_i$ and $b \in \widehat{B}$. Hence $c_h : P \rightarrow \; ^b P \rightarrow \; ^{gb} P  =\; ^h P$, where $P,\;^b P,\;^h P$ are subgroups of $S$, showing that $c_h$ is indeed a morphism in $\Gg_i$.

To show that $\Gg_{ij}$ is saturated, we will prove that $\Gg_{ij} = \Ff_S(H_{ij})$ where $H_{ij} : = \langle \wG_i, \wG_j \rangle \widehat{B}$. Set $\tG_{ij} : = \langle \wG_i, \wG_j \rangle$ and observe that since $B \le G_i \cap G_j$, the subgroup $\tG_{ij}$ is normal in $G_{ij}$. Thus $H_{ij}$ is indeed a subgroup of $G_{ij}$. Next, notice that $H_{ij} = \langle \wG_i, \wG_j \rangle \widehat{B} = \langle \wG_i \widehat{B}, \wG_j \widehat{B} \rangle = \langle H_i, H_j \rangle$. Hence $\Gg_{ij} \subseteq \Ff_S(H_{ij})$. To prove the opposite inclusion, let $h \in H_{ij}$ and $P \le S$ with $^hP \le S$. Write $h = gb$ with $g \in \tG_{ij}$ and $b \in \widehat{B}$ and argue as before.
\end{proof}

\begin{lem}\label{parhat}
Maintain the notations from above. The fusion systems $(\Gg_i; i \in I)$ form a parabolic family in $\wFf$.
\end{lem}

\begin{proof}
It follows from Lemma \ref{sathatg} that $\Gg_i$ and $\Gg_{ij}$, for all $i, j \in I$, are saturated, constrained subsystems of $\wFf$, and contain the subsystem $\wBb = N_{\wFf}(S)$.

{\it Step 1}: Each $\Gg_i$ has essential rank one.\\
According to Proposition \ref{puigprop}(iv), $Q$ is $\Gg_i$-essential if and only if $Q$ is $\wFf_i$-essential if and only if $Q$ is $\Ff_i$-essential. Since $\Ff_i$ has essential rank one, $\Gg_i$ also has essential rank one.

{\it Step 2}: $\wBb$ is a proper subsystem of $\Gg_i$.\\
Assume by contradiction that $\wFf_i \subseteq \wBb$, for some $i \in I$ and recall that $\Bb \subsetneq \Ff_i$. Then by Corollary \ref{frattini}, $\Ff_i = \langle \wFf_i, \Bb \rangle = \Bb$, a contradiction.

{\it Step 3}: The collection $\Gg_i, i \in I$ is a minimal generating set for $\wFf$.\\
Assume by contradiction that $\wFf$ is generated by $\Gg_j, j \in J$ with $J$ a proper subset of $I$. It follows from Corollary \ref{frattini} that $\Ff_i = \langle \wFf_i, \Bb \rangle$. Thus we have the following equalities:
\begin{align*}
\Ff =& \langle \Ff_i, i \in I \rangle= \langle \langle \wFf_i, \Bb \rangle, i \in I \rangle = \langle \wFf_i, i \in I; \Bb\rangle = \langle \wFf, \Bb \rangle\\
 =& \langle \wFf_j, j \in J; \Bb\rangle = \langle \langle \wFf_j, \Bb \rangle, j \in J \rangle = \langle \Ff_j, j \in J\rangle
 \end{align*}
which contradicts the fact that $\Ff_i, i \in I$ is a minimal generating set for $\Ff$.

{\it Step 4}: The following inclusions hold:
$$\wBb \subseteq \Gg_i \cap \Gg_j \subseteq (\Ff_i \cap \Ff_j ) \cap \wFf = \Bb \cap \wFf \subseteq N_{\wFf}(S) = \wBb$$
which verify (F3).

Hence the set $(\Gg_i, i \in I)$ forms a parabolic family in the fusion system $\wFf$.
\end{proof}

\begin{num}\label{thehat}
Let $\widehat{B}, H_i$ and $H_{ij}$, $i, j \in I$ be the groups constructed in \ref{sathatg} and \ref{parhat}. Set $\wG : = \langle H_i; i \in I \rangle \le G$ and observe that
\begin{align*}
\wG = &\langle H_i; i \in I \rangle = \langle \wG_i \widehat{B}; i \in I \rangle = \langle \wG_i, \widehat{B}; i \in I \rangle =\\
&\langle \wG_i, \langle N_{\wG_j}(S); j \in I \rangle; i \in I \rangle = \langle \wG_i, N_{\wG_i}(S); i \in I \rangle = \langle \wG_i; i \in I  \rangle
\end{align*}
Recall that $\Cc = \Cc(G; B, G_i, i \in I)$ denotes the chamber system associated to $\Ff$ and let $\wCc = \Cc(\wG, \widehat{B}, H_i, i \in I)$ be the chamber system associated to $\wFf$.
\end{num}

We can now formulate the main result of this section:

\begin{thm}\label{cchat}
Let $(\Ff, \Cc)$ be a fusion-chamber system pair. Let $\wFf$ be defined as in \ref{fhat} and let $\wG$ and $\wCc$ be as in \ref{thehat}. If $O_p(\wFf)=1$ then $(\wFf, \wCc)$ is a fusion-chamber system pair. Further, assume that the following two conditions hold:
\begin{list}{\upshape\bfseries}
{\setlength{\leftmargin}{1cm}
\setlength{\labelwidth}{1cm}
\setlength{\labelsep}{0.2cm}
\setlength{\parsep}{0.5ex plus 0.2ex minus 0.1ex}
\setlength{\itemsep}{0cm}}
\item[(a).] $\wCc^P$ is connected for all subgroups $P$ of $\wG$;
\item[(b).] If $P$ is $\wFf$-centric and if $R$ is a $p$-subgroup of $\Aut_{\wG}(P)$ then $(\wCc^P/C_{\wG}(P))^R$ is connected.
\end{list}
Then:
\begin{list}{\upshape\bfseries}
{\setlength{\leftmargin}{1cm}
\setlength{\labelwidth}{1cm}
\setlength{\labelsep}{0.2cm}
\setlength{\parsep}{0.5ex plus 0.2ex minus 0.1ex}
\setlength{\itemsep}{0cm}}
\item[${\rm(i).}$] $\wFf = \Ff_S(\wG)$ is a saturated fusion system that is normal in $\Ff$;
\item[${\rm(ii).}$] The map $\vp : \wCc \rightarrow \Cc$ given by $\vp(g\widehat{B}) = gB$, for $g \in \wG$, is a $2$-covering of chamber systems.
\end{list}
\end{thm}

\begin{proof}
By Lemma \ref{parhat} the set $\{ \Gg_i\;; i \in I \}$ forms a family of parabolic subsystems in $\Ff$; this gives rise to the collection of groups $\{ \widehat{B}, H_i\;; i \in I \}$ which is a parabolic system in $\wG$, according to Lemma \ref{fusch}. Recall that, by construction, the groups $B$ and $H_i$, for $i \in I$, are subgroups of $\wG$. Hence $(\Ff, \wCc)$ is a fusion-chamber system pair. 

(i). The saturation of $\wFf$ as well as the fact that $\wFf$ is realized by $\wG$ follow from an application of Theorem \ref{mythm}. To prove that $\wFf \trianglelefteq \Ff$, recall that $\Ff = \langle \Ff_i, i \in I \rangle = \langle \wFf_i, \Bb, i \in I \rangle = \langle \wFf, \Bb \rangle$. We have to show that $\Bb$ normalizes $\wFf$. But $\wFf_i \trianglelefteq \Ff_i \supseteq \Bb$ and since $\Bb$ normalizes each $\wFf_i$, it follows that $\Bb$ indeed normalizes $\wFf$.

(ii). The second part of the Theorem follows from \cite[Lemma 2.5]{timm87}, for completeness we provide the proof\footnote{See Section 3 for chamber systems and their covers.}. We have that (recall \ref{the hat}):
$$G = \langle G_i; i \in I \rangle = \langle \wG_iB; i \in I \rangle = \langle \wG_i; i \in I \rangle B = \wG B$$
thus the map $\vp: g \widehat{B} \mapsto gB$, for $g \in \wG$, is well-defined and surjective. Let $g \widehat{B}$ and $h \widehat{B}$ be two $i$-adjacent chambers in $\wCc$. Then $h^{-1}g \in H_i$. Since $H_i \le G_i$, it follows that $gG_i = hG_i$ and the chambers $gB$ and $hB$ are $i$-adjacent in $\Cc$. Thus $\vp$ is a morphism of chamber systems.

To show that $\vp$ is a $2$-covering, we must prove that $\vp_{ij}$, the restriction of $\vp$ to a rank two $\{i, j\}$-residue is bijective, for each $i, j \in I$. Clearly $\vp_{ij}$ is surjective. As $\Bb= N_{\Ff}(S) = N_{\Ff_{ij}}(S)$, and similarly $\wBb = N_{\wFf}(S) = N_{\Gg_{ij}}(S)$, we may use \cite[Proposition 3.8]{lib1}, Lemma \ref{pprimep} and Theorem \ref{thmconstraint}, to write $B = N_{G_{ij}}(S)$ and $\widehat{B} = N_{H_{ij}}(S)$. Thus $\widehat{B} = B \cap H_{ij}$. An application of the Frattini argument, together with the fact that $H_{ij} = \tG_{ij} \widehat{B}$ give that $G_{ij} = \tG_{ij} B = H_{ij} B$. The product formula gives:
$$|H_{ij}B| \cdot |H_{ij} \cap B| = |H_{ij}B| \cdot |\widehat{B}|  = |G_{ij}| \cdot |\widehat{B}| = |H_{ij}| \cdot |B|$$
which shows that the $(ij)$-residues in $\wCc$ and $\Cc$ have the same number of chambers. Hence $\vp_{ij}$ is injective. This concludes our proof that $\vp$ is a $2$-covering.
\end{proof}

\section{An application: classical parabolic families in fusion systems}

In this section we investigate fusion systems which contain parabolic families with specific properties. In finite group theory, the parabolic systems of type $\Ms$, defined below, are at the core of so called amalgam method and emphasize the deep connections between local analysis, on one side, and group geometries on the other side. For a comprehensive overview of the early results on parabolic systems of type $\Ms$, see Meixner \cite{meixnerrev}. For an up to date succinct overview on higher rank amalgams see Parker and Rowley \cite[Chapter 24]{sympaml}. The terminology is inspired by the structure of the groups of Lie type, in which a parabolic system $(B, G_i; i \in I)$ consists of the Borel subgroup $B$, together with the minimal parabolic subgroups $G_i$ containing it. We will show that these notions have natural generalizations to the context of fusion systems. We start with reviewing some of the standard terminology.

\subsection*{Chamber systems and classical parabolic systems}

A chamber system $\Cc$ of rank two is a {\it generalized digon} if and only if $\Cc$ is the chamber system associated to a parabolic system of the form $\Pp(G)=(G_i \cap G_j, G_i, G_j)$ with $G = G_iG_j=G_jG_j$ and $G_i \ne G_j$. A rank two chamber system $\Cc$ is a {\it classical generalized $m_{ij}$-gon}, for $m_{ij} \ge 3$, if $\Cc$ is isomorphic to $\Cc(G; B, G_i, G_j)$ where $G$ is an (essentially\footnote{The following groups are also allowed: $A_6$, $S_6$, $G_2(2)' \simeq U_3(3)$, $G_2(2)$, $^2F_4(2)'$ and $^2F_4(2)$.}) simple rank two group of Lie type in characteristic $p$, $B=G_i\cap G_j$ is a Borel subgroup of $G$ and $G_i$ and $G_j$ are the two maximal parabolic subgroups of $G$ containing $B$; here $m_{ij}$ denotes the integer that defines the Weyl group of $G$.

\begin{num}\label{ctcs}
A connected chamber system $\Cc$ over $I$ is a {\it classical Tits chamber system} if all rank two residues are either generalized digons or classical generalized $m_{ij}$-gons, for $i, j \in I$ with $i \ne j$. Here $m_{ij}$ is fixed for each type. $\Cc$ is called {\it locally finite}, if all rank two residues are finite. The {\it diagram} (or {\it the type of} $\Cc$) is denoted $\Ms=\Ms (I)$ and it is a graph whose vertices are labeled by the elements of $I$, the nodes $i$ and $j$ are connected by a bond of strength $m_{ij}-2$.
\end{num}

\begin{num}\label{parm}
A collection of finite subgroups $(B, G_i; i \in I)$ in a group $G$ is a {\it classical parabolic system\footnote{Conform to \cite{fsw00}.}} if it is a parabolic system in the sense of \ref{parsys} and in addition, fulfills the following conditions.
\begin{list}{\upshape\bfseries}
{\setlength{\leftmargin}{1.2cm}
\setlength{\labelwidth}{1cm}
\setlength{\labelsep}{0.2cm}
\setlength{\parsep}{0.5ex plus 0.2ex minus 0.1ex}
\setlength{\itemsep}{0cm}}
\item[(i).]$S = O_p(B) \in \Syl_p(G_{ij})$ with $G_{ij} = \langle G_i, G_j \rangle$, for all $i, j \in I$.
\item[(ii).] For each $i \in I$, $G_i/O_p (G_i)$ is a rank one group of Lie type in characteristic $p$.
\item[(iii).] For each pair $i, j \in I$, either $G_{ij}=G_iG_j=G_jG_i$ or $G_{ij}/O_p(G_{ij})$ is a rank two Lie type group in characteristic $p$.
\end{list}
The associated chamber system $\Cc = \Cc(G; B, G_i, i \in I)$ is a locally finite Tits chamber system of classical type; see \cite[Lemma 4.2]{timm83I} for a proof. The diagram $\Ms$ is constructed as in \ref{ctcs}; in particular, $m_{ij}$ is $2$ if $G_{ij} = G_iG_j$ and it is determined by the Weyl group of $G_{ij}$ otherwise.
\end{num}

We end this review with the following standard result:

\begin{prop}\cite[3.1]{timm85rev}\label{buildingres}
Let $\Cc$ be a locally finite classical Tits chamber system over $I$, with $|I| \ge 3$, with spherical diagram $\Ms$ and chamber transitive automorphism group $G$. Then one of the following holds:
\vspace*{-.2cm}
\begin{list}{\upshape\bfseries}
{\setlength{\leftmargin}{1.2cm}
\setlength{\labelwidth}{1cm}
\setlength{\labelsep}{0.2cm}
\setlength{\parsep}{0cm}
\setlength{\itemsep}{0cm}}
\item[${\rm(i).}$] $\Cc$ is a finite spherical building of type $\Ms$ and $G$ is an extension of a simple group of Lie type $\Ms$ by diagonal and field automorphisms or $G \simeq A_7$ and $\Ms=A_3$.
\item[${\rm(ii).}$]$\Cc$ is the Neumeier chamber system obtained from $A_7$, $\Ms=C_3$ and $G \simeq A_7$.
\end{list}
\end{prop}

The diagram is spherical if the associated Weyl group is finite. Buildings are special cases of chamber systems of type $\Ms$, for a detailed treatment see Ronan \cite{ronanbook} or the monumental paper of Tits \cite{titslocal}. For a description of the Neumeier chamber system of type $C_3$ in the alternating group $A_7$ see \cite[Example 2, pg. 50]{ronanbook}.

\subsection*{Classical parabolic families in fusion systems}
In this section we apply Theorem \ref{mythm} to a fusion system $\Ff$ which contains a classical family of parabolic systems of type $\Ms$, as defined below. We shall use the notation introduced at the beginning of Section $5$. In particular, recall that $\Ff$ contains a collection of saturated constrained fusion subsystems $\{\Ff_i\;; i \in I\}$, and for each pair $i, j \in I$, the subsystems $\Ff_{ij} = \langle \Ff_i, \Ff_j \rangle$ are also saturated and constrained. Also $\Bb = N_{\Ff}(S)$. We set $U_i = O_p(\Ff_i)$ and $U_{ij} = O_p(\Ff_{ij})$ for all $i, j \in I$.

\begin{defn}\label{classicalfps}
Let $\Ff$ be a fusion system over a finite $p$-group $S$. We say that $\Ff$ contains a {\it classical family of parabolic systems of type $\Ms$} if the following hold.
\vspace*{-.2cm}
\begin{list}{\upshape\bfseries}
{\setlength{\leftmargin}{1.2cm}
\setlength{\labelwidth}{1cm}
\setlength{\labelsep}{0.2cm}
\setlength{\parsep}{0cm}
\setlength{\itemsep}{0cm}}
\item[${\rm(i).}$] $\Ff$ contains a family of parabolic systems, in the sense of \ref{parffus}.
\item[${\rm(ii).}$] For each $i \in I$, $\Out_{\Ff_i}(U_i)$ is a rank one finite group of Lie type in characteristic $p$.
\item[${\rm(iii).}$] For each pair $i,j \in I$, $\Out_{\Ff_{ij}}(U_{ij})$ is either a rank two finite group of Lie type in characteristic $p$ or it is a (central) product of two rank one finite groups of Lie type in characteristic $p$.
\end{list}
\vspace*{-.1cm}
To such a fusion system we can associate a diagram $\Ms$ on $I$ in the following way: if $\Out_{\Ff_{ij}}(U_{ij})$ is a product of two rank one groups of Lie type then $i$ and $j$ are not connected, if $\Out_{\Ff_{ij}}(U_{ij})$ is a rank two group of Lie type we take the corresponding Coxeter diagram for the edge between $i$ and $j$.
\end{defn}

We arrive at the main result of this Section, a generalization of Proposition \ref{buildingres} to fusion systems.
\begin{prop}\label{mtype}
Let $(\Ff, \Cc)$ be a fusion-chamber system pair with $|I| \ge 3$. Assume that:
\begin{list}{\upshape\bfseries}
{\setlength{\leftmargin}{1.3cm}
\setlength{\labelwidth}{1cm}
\setlength{\labelsep}{0.2cm}
\setlength{\parsep}{0.5ex plus 0.2ex minus 0.1ex}
\setlength{\itemsep}{0cm}}
\item[${\rm(i).}$] $\Ff$ contains a classical family of parabolic systems of type $\Ms$;
\item[${\rm(ii).}$] $\Ms$ is a spherical diagram.
\end{list}
Then $\Ff$ is the fusion system of a finite simple group of Lie type in characteristic $p$ extended by diagonal and field automorphisms.
\end{prop}

\begin{proof}
Let $(\Ff, \Cc)$ be a fusion-chamber system pair as defined in \ref{cf}. First recall that $\Ff$ contains a family of parabolic subsystems (see \ref{parffus}). There is a collection of $p'$-reduced, $p$-constrained finite groups $B$, $G_i$ and $G_{ij}$, that realize $\Bb$, $\Ff_i$ and $\Ff_{ij}$ respectively. By Proposition \ref{diagfus}, these groups (together with appropriate group homomorphisms) form a diagram of groups (as defined in \ref{amalgam}). If $G$ is a faithful completion of this diagram of groups then $\{ B, G_i; i \in I\}$ is a parabolic system in $G$, by Lemma \ref{fusch}.

For each $i \in I$, the fusion system $\Ff_i=\Ff_S(G_i)$ is saturated and constrained, thus $U_i$ is $\Ff_i$-centric and $U_i = O_p(G_i)$. An application of \cite[Proposition III.5.8]{ako} gives that $\Out _{\Ff_i} (U_i)  \simeq G_i /U_i$. Similarly, we obtain $\Out_{\Ff_{ij}}(U_{ij}) = G_{ij}/U_{ij}$ and $U_{ij} = O_p(G_{ij})$. Also recall that $G_{ij} = \langle G_i, G_j\rangle$, conform to \ref{gijgenerated}.

Let now $\Cc = \Cc (G; B, G_i, i \in I)$. Since $G =\langle G_i, i \in I \rangle$, the chamber system $\Cc$ is connected. By transitivity each $\{i,j\}$-residue of $\Cc$ is isomorphic to $\{ gB | \; gB \subseteq G_{ij} \}$. Since the parabolic family in $\Ff$ is of classical type, (ii) and (iii) in Definition \ref{classicalfps} imply (ii) and (iii) from \ref{parm}. As \ref{parm}(i) clearly holds, it follows that $\{ B, G_i\;; i \in I \}$ is a classical parabolic system in $G$. This implies that $\Cc$ is a locally finite Tits chamber system of classical type. Since $\Ms$ is spherical, according to Timmesfeld results \cite[3.14, 4.3]{timm83I} and \cite[3.1]{timm85rev}, see also Proposition \ref{buildingres}, $\Cc$ is a finite spherical building of type $\Ms$ and $G$ is an extension of a simple group of Lie type $\Ms$ by diagonal and field automorphisms, or $G \simeq A_7$ and $\Ms$ is of type $A_3$, or $\Cc$ is the Neumeier chamber system obtained from $A_7$, $\Ms=C_3$ and $G \simeq A_7$.

In the former case with $\Cc$ a building, it is well known that $\Cc^P$ are contractible for all subgroups $P \le S$; an elegant proof of this fact can be found in \cite[proof of Theorem 3.1]{qu78}. The claim follows by an application of Proposition \ref{ffsg}.

Let now $G \simeq A_7$ and let $B$ denote some Sylow $2$-subgroup of $G$. Then $G$ has exactly four subgroups $X_1, \ldots , X_4$ containing $B$ and which are isomorphic to $S_4$ (the symmetric group on four letters). For suitable labeling, see \cite[Example 1.2]{meixnerrev}, we get:
\vspace*{-.2cm}
\begin{list}{\upshape\bfseries}
{\setlength{\leftmargin}{1cm}
\setlength{\labelwidth}{1cm}
\setlength{\labelsep}{0.2cm}
\setlength{\parsep}{0cm}
\setlength{\itemsep}{0cm}}
\item[a.] $\Cc_i = \Cc(G; B, X_1, X_2, X_i)$, for $i=3,4$, two isomorphic Neumeier chamber systems of type $C_3$; \item[b.] $\Cc'= \Cc(G; B, X_3, X_2, X_4)$, the chamber system of type $A_3$ over the field with two elements.
\end{list}
\vspace{-.1cm}
Hence, in either case $\Bb = \Ff_{D_8}(D_8)$ and $\Ff_i = \Ff_{D_8}(S_4)$ for $i = 1, 2, 3$. The three fusion subsystems $\Ff_i$, do not form a minimal generating set for $\Ff_{D_8}(A_7)$; only two subsystems are sufficient to generate $\Ff$ (see \cite[Example I.2.7]{ako}). Thus, we do not obtain families of parabolic subsystems for $\Ff_{D_8}(A_7)$ (in the sense of \ref{parffus}).
\end{proof}

\bibliographystyle{alpha}
\bibliography{refer}

\end{document}